%% file: Eilenberg.tex
\documentclass[oneside,article,12pt]{memoir}
\usepackage[margin=2.5cm]{geometry}
\usepackage{amsthm}
\usepackage{amssymb}
\usepackage{amsmath}
\usepackage{mathtools}
\usepackage{faktor}
\usepackage{relsize}
\usepackage{graphicx,caption,color}
\usepackage[dvipsnames]{xcolor}
\usepackage{babel}
\usepackage{hyperref}
\usepackage{cite}
\usepackage{environ}
\usepackage{ifthen}
\usepackage{tikz}
\usepackage{tikz-cd}
\usepackage{tikz-3dplot}

\usetikzlibrary{calc}
\tdplotsetmaincoords{70}{30}
\usetikzlibrary{positioning}

\tikzset{every node/.append style={fill=white,font=\footnotesize,fill opacity=1,text opacity=1,draw opacity=1}}
\tikzset{roundnode/.style={circle, draw=darkgray, fill=blue!5, very thick,inner sep=0pt, minimum size=3mm}}

\counterwithout{section}{chapter}
\nobibintoc

\input{Makros}

\def\gref{\psi}
\def\cref{\hat{\psi}}
\def\tref{|\cref|}
\def\grefl{\lambda}
\def\crefl{\hat{\lambda}}

\def\grefr{\rho}
\def\crefr{\hat{\rho}}

\def\AA{\mathcal{A}}
\def\zero{\mathbf{0}}
\def\hidx{J}
\def\hcoord{T}
\def\chop{r}
\def\magic{\alpha}
\def\Hurewicz{h}
\def\relHurewicz{\bar{h}}
\def\classHurewicz{h_{cl}}

\input{MakroThmCounter}
\title{Eilenberg-MacLane spaces in A-homotopy theory}
\author{Georg Wille}
\date{\today}
\begin{document}

\maketitle

\begin{abstract}
    In this paper we study a discrete homotopy theory for undirected graphs.
    More precisely, we consider $A$-homotopy groups, first defined in the work of Barcelo, Kramer, Laubenbacher and Weaver.
    For every degree $n$ and every group $G$ (Abelian if $n\geq 2$), we construct a graph whose $n$-th $A$-homotopy group is isomorphic to $G$ and whose other $A$-homotopy groups all vanish. Using these, we show that for every infinite sequence of groups $G_1,G_2,\ldots$, with $G_i$ Abelian for $i\geq 2$, there is a graph $X$ such that the $i$-th $A$-homotopy group of $X$ is isomorphic to $G_i$ for all $i$. Along the way, we identify relative $A$-homotopy groups with relative classical homotopy groups, and prove a discrete version of the Blakers-Massey theorem. We then construct partial Postnikov systems and (full) Whitehead towers.
\end{abstract}
\tableofcontents*
\section{Introduction}
Regarding a graph as a one-dimensional CW-complex, topological invariants such as homotopy and homology groups offer little insight into the structure of the graph. In fact, every connected graph is homotopy equivalent to a wedge of circles. In recent years, several alternative homotopy
and homology theories for
graphs have been defined and studied. The results show
that they capture much
more detailed information about the graphs (see\cite{Hepworth:Magnitude,Yau:Path,Helene:Perspectives,Dochtermann:Cross}). In general however, their expressivity
is still wide open and
part of the motivation for the work in this paper. Here, we concentrate
on $A$-homotopy theory,
which is based on the theory of $Q$-analysis by Atkin
\cite{Atkin:I,Atkin:II}. More precisely, we consider $A$-homotopy groups introduced in \cite{Helene:Foundations} and discrete cubical homology introduced in \cite{Helene:Homology}. Only recently in \cite{Chris:Revisited}, it was shown that these notions of homology and homotopy groups are compatible in the sense that they satisfy a corresponding version of the Hurewicz theorem and it was shown that the $A$-homotopy groups of a graph $X$ are naturally isomorphic to the classical homotopy groups of a CW-complex $|N_1X|$.

The interesting question remains how rich the theory is, namely: ``Does every phenomenon observed in the classical topology have a counterpart in
$A$-homotopy theory?''. This question is formalized in the \fa{discrete homotopy hypothesis}, which states that the localization of the category of graphs along those maps which induce isomorphisms on $A$-homotopy groups is equivalent to the classical homotopy category, that is the category of CW-complexes with morphisms being homotopy classes of continuous maps. Recently, a major advancement towards this conjecture has been made in \cite{Chris:ntypes}, where it was shown that the localization along those graph maps, which induce isomorphisms on the first $n$ homotopy groups is equivalent to the category of CW-complexes localized at those maps that induce isomorphisms on the first $n$ classical homotopy groups. In particular, they show, that for every finite sequence of groups $G_1,\ldots ,G_n$, with $G_i$ Abelian for $i\geq 2$, there is a graph $X$ such that the $i$-th discrete homotopy group of $X$ is isomorphic to $G_i$ for $1\leq i\leq n$. A key observation of theirs is that the fact that $A$-homotopy groups of a graph $X$ are determined by the $(n+1)$-skeleton of $|N_1X|$, can be used to control the first $n$ $A$-homotopy groups via $(n+1)$-\fa{skeletal pushouts}.

This work is heavily influenced by their observation, however we offer a different perspective here by working with relative $A$-homotopy groups. Relative $A$-homotopy groups were also introduced in \cite{Helene:Foundations}, we will show that for a pair of graphs $(X,A)$, it holds that the relative $A$-homotopy groups of $(X,A)$ are isomorphic to the relative classical homotopy groups $\pi_*(|N_1X|,|N_1A|)$. We do this by directly constructing a group homomorphism $$\AA_*(X,*)\to \pi_*(|N_1X|,*)$$ and corresponding relative version of this. While the existence of such an isomorphism has been shown (in the non-relative case) in \cite{Chris:Revisited}, we provide an explicit description of this map, although the author believes that this description is unlikely to surprise experts. The isomorphism $\AA_*(X,A,*)\simeq \pi_*(|N_1X|,|N_1A|,*)$ then allows us to give an $A$-homotopy version of the Blakers-Massey theorem.

Our main result is to show the existence of \fa{Eilenberg-Mac\-Lane graphs}, that is, for every degree $n$ and every group $G$ (with $G$ Abelian if $n\geq 2$) there is a graph $X$ such that the $n$-th $A$-homotopy group of $X$ is isomorphic to $G$ and all other $A$-homotopy groups vanish. For $n=1$ we will give an explicit description of these very similar to Milnor's construction of classifying spaces $B\mathcal{G}$ for a topological group $\mathcal{G}$ (\cite{Milnor:BundlesI},\cite{Milnor:BundlesII}). To construct Eilenberg-MacLane graphs in higher degrees, we will show the existence of partial Postnikov systems, partial in the sense that the tower starts at a fixed $n$ and goes downwards. To the author's surprise however, this partial construction is enough to yield full Whitehead towers. Finally, we show that also for every infinite sequence of groups $G_1,G_2,\ldots$, Abelian for $i\geq 2$, there is a graph $X$ such that the $i$-th $A$-homotopy group of $X$ is isomorphic to $G_i$ for all $i$.

In Section 2 we will give an overview of the previously known results in $A$-homotopy theory that will be used in the rest of the paper. In section 3 we will prove the above mentioned isomorphism $\AA_*(X,A,*)\simeq \pi_*(|N_1X|,|N_1A|,*)$. Finally, in Section 4, we construct Eilenberg-MacLane graphs as well as partial Postnikov systems and (full) Whitehead towers.

\textbf{Acknowledgements:}
The author would like to thank H{\'e}l{\`e}ne Barcelo, Curtis Greene, Chris Kapulkin and Volkmar Welker for insightful discussions and their comments on an earlier version of this paper.

\section{Preliminaries}
We will give an overview of $A$-homotopy theory of graphs, including (almost) all of the results and definitions needed in the rest of the article. The author claims no originality on any of this.

Notational conventions: In this article, the letters $X,Y,Z$ will denote graphs and $A,B$ will denote subgraphs; groups will be denoted $G,M$. Sans-serif letters $\mathsf{X},\mathsf{Y}$ denote cubical sets.
For us, a graph $X=(V(X),E(X))$ is a set of vertices $V(X)$ with a set of edges $E(X)\subseteq \binom{V(X)}{2}$, i.e. a graph is a one-dimensional simplicial complex.

For a vertex $v\in V(X)$, the \fa{$1$-neighborhood of $v$ in $X$}, denoted by $\Umg{1}{v}$, is the smallest (not necessarily induced) subgraph containing $v$ and all edges $\{v,w\}\in E(X)$. Similarly, the \fa{$2$-neighborhood of $v$ in $X$}, denoted by $\Umg{2}{v}$, is the union of $\Umg{1}{w}$ for all vertices $w\in V(\Umg{1}{v})$.
\begin{definition}
    Let $X$ and $Y$ be graphs and $f:V(X)\to V(Y)$ a map between vertex sets. If $f$ satisfies that for all edges $\{x_1,x_2\}\in E(X)$ it holds that $f(x_1)=f(x_2)$ or $\{f(x_1),f(x_2)\}\in E(Y)$, then $f$ is called a \fa{graph map} and denoted $f:X\to Y$.
\end{definition}

The following examples of graphs play a central role in the theory; we use the notation $[m]=\{0,1,2,\ldots, m\}$:
\begin{example}
    \begin{itemize}
        \item The \fa{interval graph} $I_m$ with vertices $[m]$ and edges $\{k,k+1\}$ for all $k\in[m-1]$.
              
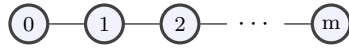
\begin{center}

                  \begin{tikzpicture}
                      \tikzset{roundnode/.style={circle, draw=darkgray, fill=blue!5, very thick,inner sep=0pt, minimum size=5mm}}
                      \node[roundnode](0) at (0,0){\smaller 0};
                      \node[roundnode](1) at (1,0){\smaller 1};
                      \node[roundnode](2) at (2,0){\smaller 2};
                      \node[roundnode](m) at (4,0){\smaller m};
                      \node (dots) at (3,0){$\cdots$};
                      \draw (0.east) -- (1.west);
                      \draw (1.east) -- (2.west);
                      \draw (2.east) -- (dots);
                      \draw (dots) -- (m.west);
                  \end{tikzpicture}
              \end{center}

        \item The \fa{infinite interval graph} $I_\infty$ with vertices $\mathbb{Z}$ and edges $\{k,k+1\}$ for $k\in\ZZ$
              \begin{center}
                  \begin{tikzpicture}
                      \tikzset{roundnode/.style={circle, draw=darkgray, fill=blue!5, very thick,inner sep=0pt, minimum size=5mm}}
                      \node[roundnode](-2) at (0,0){\smaller -2};
                      \node[roundnode](-1) at (1,0){\smaller -1};
                      \node[roundnode](0) at (2,0){\smaller 0};
                      \node[roundnode](1) at (3,0){\smaller 1};
                      \node[roundnode](2) at (4,0){\smaller 2};
                      \draw(-2.east) -- (-1.west);

                      \draw (-1.east) -- (0.west);
                      \draw (0.east) -- (1.west);
                      \draw (1.east) -- (2.west);
                      \node (ldots) at (-1,0){$\cdots$};
                      \node (rdots) at (5,0){$\cdots$};
                      \draw (ldots)--(-2.west);
                      \draw (rdots)--(2.east);
                  \end{tikzpicture}
              \end{center}
        \item The \fa{$m$-cycle graph} $C_m$ with vertices $[m-1]$ and edges $\{k,k+1\}$ for all $k\in[m-2]$ as well as one edge connecting $0$ and $m-1$
              
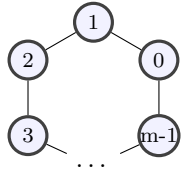
\begin{center}
                  \begin{tikzpicture}
                      \tikzset{roundnode/.style={circle, draw=darkgray, fill=blue!5, very thick,inner sep=0pt, minimum size=5mm}}
                      \node[roundnode] (0) at (30:1cm) {\smaller 0};
                      \node[roundnode] (1) at (90:1cm) {\smaller 1};
                      \node[roundnode] (2) at (150:1cm) {\smaller 2};
                      \node[roundnode] (3) at (210:1cm) {\smaller 3};
                      \node (4) at (270:0.9cm) {\dots};
                      \node[roundnode] (5) at (330:1cm) {\smaller m-1};
                      \draw (0)--(1)--(2)--(3)--(4)--(5)--(0);
                  \end{tikzpicture}
              \end{center}
    \end{itemize}
\end{example}
There are several products on the category of graphs, we will need the following.
\begin{definition} Let $X$ and $Y$ be graphs.
    \begin{itemize}
        \item The \fa{disjoint union of graphs} $X\sqcup Y$ has both vertices and edges the disjoint union of the vertices and edges of $X$ and $Y$ respectively.
        \item The \fa{box product} $X \bprod Y$ is the graph with vertices $V(X\bprod Y)=V(X)\times V(Y)$ and edges
              \begin{align*}
                  \big\{ \left\{ (x, y_1), (x, y_2) \right\} \,\big|\, x \in V(X), \left\{ y_1, y_2 \right\} \in E(Y) \big\}\cup\big\{ \left\{ (x_1, y), (x_2, y) \right\} \,\big|\, \left\{ x_1, x_2 \right\} \in E(X), y \in V(Y) \big\}
              \end{align*}
        \item The \fa{categorical product} $X\times Y$ is the graph containing all vertices and edges like $X \bprod Y$ with the additional edges
              \begin{align*}
                  \big\{ \left\{ (x_1,y_1),(x_2,y_2) \right\} \,\big|\,\left\{ x_1, x_2 \right\} \in E(X),\left\{ y_1, y_2 \right\} \in E(Y) \big\}
              \end{align*}
    \end{itemize}
\end{definition}
The following figure illustrates the different products.
\medskip

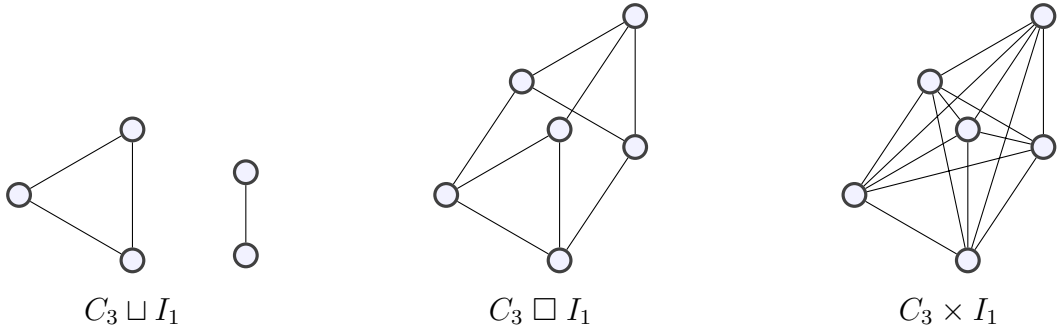
\begin{minipage}[t]{0.32\textwidth}
    \centering
    \begin{tikzpicture}
        \node[roundnode](0D) at ($(180:1cm)+(0,0.5)$) {};
        \node[roundnode](1D) at ($(60:1cm)+(0,0.5)$) {};
        \node[roundnode](2D) at ($(300:1cm)+(0,0.5)$) {};
        \draw (0D) -- (1D) -- (2D) -- (0D);
        \node[roundnode](0I) at (2,-0.3) {};
        \node[roundnode](1I) at (2,0.8) {};
        \draw (0I) -- (1I);
    \end{tikzpicture}
    \captionof*{figure}{$C_3\sqcup I_1$}
\end{minipage}
\begin{minipage}[t]{0.32\textwidth}
    \centering
    \begin{tikzpicture}
        \node[roundnode](0) at ($(180:1cm)+(-1,-1.5)$) {};
        \node[roundnode](1) at ($(60:1cm)+(-1,-1.5)$) {};
        \node[roundnode](2) at ($(300:1cm)+(-1,-1.5)$) {};
        \node[roundnode](3) at (180:1cm) {};
        \node[roundnode](4) at (60:1cm) {};
        \node[roundnode](5) at (300:1cm) {};
        \draw (0) -- (1) -- (2) --(0);
        \draw (3) -- (4) -- (5) -- (3);
        \draw (0) -- (3);
        \draw (1) -- (4);
        \draw (2) -- (5);
    \end{tikzpicture}
    \captionof*{figure}{$C_3\bprod I_1$}
\end{minipage}
\begin{minipage}[t]{0.32\textwidth}
    \centering
    \begin{tikzpicture}
        \node[roundnode](0) at ($(180:1cm)+(-1,-1.5)$) {};
        \node[roundnode](1) at ($(60:1cm)+(-1,-1.5)$) {};
        \node[roundnode](2) at ($(300:1cm)+(-1,-1.5)$) {};
        \node[roundnode](3) at (180:1cm) {};
        \node[roundnode](4) at (60:1cm) {};
        \node[roundnode](5) at (300:1cm) {};
        \draw (0) -- (1) -- (2) --(0);
        \draw (3) -- (4) -- (5) -- (3);
        \draw (0) -- (3);
        \draw (1) -- (4);
        \draw (2) -- (5);
        \draw (0) -- (4) -- (2) -- (3)--(1)--(5)--(0);
    \end{tikzpicture}
    \captionof*{figure}{$C_3\times I_1$}
\end{minipage}
\medskip

With all of these products the category of graphs can be made into a monoidal category (see \cite{Monoidalproducts}); in particular all of these are associative up to natural isomorphism. For a graph $X$ we will use $X^n$ to denote the $n$-fold box product of $X$ with the convention that $X^0=I_0$, the graph with one vertex. To justify this convention observe that $X\bprod I_0 \cong X$.
\begin{definition}
    For natural numbers $m,n$, denote by
    \begin{itemize}
        \item a \fa{facet} $\del_i^{\ve}(I_m^n)$, for $\ve\in\{0,1\}$ and $1\leq i\leq n$, i.e. the induced subgraph of $I_m^n$ with vertices
              \[
                  \{(v_1,\ldots, v_n)\in V(I_m^n)\mid v_i= \ve\cdot m\},
              \]
        \item $\del I_m^n$ the \fa{boundary of the cube $I_m^n$}, that is the subgraph induced by the set of vertices
              \[
                  \{(v_1,\ldots, v_n)\in V(I_m^n)\mid v_i=0 \text{ or } v_i=m \text{ for some }1\leq i\leq n \},
              \]
              put differently
              \[
                  \del I_m^n= \bigcup_{i=1}^n\del_i^0(I_m^n)\cup \del^1_i(I_m^n)
              \]
        \item and $\hg{m}{n}=\del I_m^{n-1}\bprod I_m \cup I_m^{n-1}\bprod\{0\}$ the \fa{$n$-horn}.
    \end{itemize}
\end{definition}
When we take quotients, like in the following definition, by identifying vertices via an equivalence relation, there will be an edge between two vertices whenever there is an edge between two representatives.
\begin{definition}
    Given a graph $X$, let the $m$-\fa{suspension} of $X$, denoted by $\Sigma_m X$, be the graph
    \[
        \Sigma_m X \letbe (X\bprod I_m) / \sim,
    \]
    where $\sim$ is the equivalence relation generated by $(x,0)\sim (y,0)$ and $(x,m)\sim(y,m)$ for all $x,y\in X$. Let $\Sigma_m^nX$ denote the $n$-\fa{fold $m$-suspension} of $X$, i.e. $\Sigma_m^nX = \Sigma_m(\Sigma_m^{n-1}X)$.
\end{definition}

\subsection{Discrete homotopy groups}
\begin{definition}~
    Let $X,Y$ be graphs.
    \begin{itemize}
        \item Let $f,g\from X\to Y$ be graph maps. A graph map $H:X\bprod I_m \to Y$ is called a \fa{homotopy from $f$ to $g$} if for all $x\in X$ it holds that
              \[
                  H(x,0)=f(x)\text{ and } H(x,m)=g(x).
              \]
        \item For graphs $X$ and $Y$ and subgraphs $A\subset X$ and $B\subset Y$, a map of pairs $$f:(X,A)\to (Y,B)$$ is a graph map $f:X\to Y$ such that $f(A)\subset B$.
        \item Let $*_X\in V(X)$. A graph map $f:I_\infty^n\to X$ is called \fa{almost everywhere constant (a.~e.~c.) with basepoint $*_X$} if all but finitely many vertices in $I_\infty^n$ get mapped to $*_X$.
        \item Given two almost everywhere constant maps $f,g:I_\infty^n\to X$ with basepoint $*_X$, an \fa{almost everywhere constant homotopy} is a homotopy $H:I_\infty^n\bprod I_m \to X$ from $f$ to $g$ such that all but finitely many vertices get mapped to $*_X$.
    \end{itemize}
\end{definition}

For all graphs $X$ and $Y$, homotopy induces an equivalence relation on the set of all graph maps $X \to Y$.
The pair $(X,*_X)$ of a graph $X$ and a vertex $*_X\in V(X)$ is called a \fa{based graph}.
If it is clear from the context with respect to which basepoint maps and homotopies are a.e.c, we will omit it from notation.

Given two almost everywhere constant maps $f,g:I_\infty^n\to X$, since all but finitely many vertices are mapped to $*_X$ there is an integer $M$ such that for all $(v_1,\ldots, v_n)\in I_\infty^n$ it holds that
\begin{itemize}
    \item if $v_1\geq M$ then $f(v_1,\ldots, v_n)=*_X$
    \item and if $v_1\leq -M$ then $g(v_1,\ldots, v_n)=*_X$.
\end{itemize}
Using the smallest such $M$, consider the \fa{concatenation} $f\cdot g$ as the following a. e. c. map
\begin{align*}
    f\cdot g \colon I_\infty^n & \la X                                                \\
    (v_1,v_2,\ldots, v_n)      & \lm \begin{cases}
                                         f(v_1+M,v_2,\ldots, v_n) & \text{ for } v_1\leq0 \\
                                         g(v_1-M,v_2,\ldots, v_n) & \text{ for } v_1\geq0
                                     \end{cases}
\end{align*}
If there are a. e. c. homotopies from $f,g$ to $f',g'$ respectively, then there is also a a. e. c. homotopy from $f\cdot g$ to $f'\cdot g'$. Concatenation of $f$ with the constant a. e. c. map yields $f$ up to precomposition with a translation and thus this concatenation is homotopic to $f$. Consider the map $-f$ defined as $-f(v_1,v_2\ldots v_n)\letbe f(M-v_1,v_2\ldots v_n)$, then $f\cdot -f$ is homotopic to the constant map.
All of this leads to the following definition.

\begin{definition}\label{dfn:htpygroups}
    For every $n\geq 1$, the \fa{$n$-th discrete homotopy group} is the set of equivalence classes of almost everywhere constant graph maps
    \begin{align*}
        \AA_n(X,*_X)=\left\{ f\colon I_\infty^n\to X \;\middle|\; f\text{ a. e. c. with basepoint } *_X \right\}/\sim
    \end{align*}
    where the equivalence relation $\sim$ is almost everywhere constant homotopy. For $n\geq 1$, the group multiplication is given by concatenation via
    \[
        [f]\cdot[g]=[f\cdot g]
    \]
    and the neutral element is the constant map with basepoint $*_X$.
    We also define $\AA_0(X,*_X)$ as the set of connected components of $X$, pointed by the component containing $*_X$. A sequence of pointed sets $(A,*_A) \xrightarrow{f} (B,*_B) \xrightarrow{g} (C,*_C)$ is called \fa{exact} at $B$ if $\im(f) = g^{-1}(*_C)$.
\end{definition}
Now, for $m\geq 1$ consider the graph maps
\begin{align*}
    \grefl_m: I_{m+1} & \la I_{m}       &  & \text{and  } & \grefr_m:I_{m+1} & \la I_m        \\
    t                 & \lm \max(0,t-1) &  &              & t                & \lm \min(m,t).
\end{align*}
and denote by $\cat{LR}_n$ the subcategory of $\Graph$ with objects $I_m^n$ and morphisms generated by the $n$-fold box product $\grefl_m^n=\grefl_m\bprod \ldots \bprod \grefl_m$ and $\grefr_m^n=\grefr_m\bprod \ldots \bprod \grefr_m$ for all $m$. For $n,m\geq 1$ denote with $\zero$ the vertex $(0,\ldots,0)\in I_m^n$.
\begin{proposition}
    Let $(X,*)$ be a based graph and $n\in\NN$.
    The following are naturally isomorphic as sets:
    \begin{enumerate}
        \item $$\AA_n(X,*)$$
        \item $$\left\{f:(I_m^n,\del I_m^n)\to(X,*) \, \middle\vert\, m\in\NN\right\}/\sim$$
        \item $$\left\{f:(\del I_m^{n+1},\zero)\to(X,*) \, \middle\vert\, m\in\NN\right\}/\sim$$
    \end{enumerate}
    where in the cases (2) and (3) the equivalence relation $\sim$ is generated by both homotopies that are also maps of pairs of graphs, and precomposition with the morphisms of $\cat{LR}_n$ and $\cat{LR}_{n+1}$ respectively. Note that in case (3) we identify $\grefl_m^{n+1}$ and $\grefr_m^{n+1}$ with their restrictions $\del I_{m+1}^{n+1}\to \del I_{m}^{n+1}$.
\end{proposition}
\begin{proof}
    That (1) and (2) are isomorphic is \cite[4.6]{Chris:Revisited},
    the fact that (2) and (3) are isomorphic follows from \cite[3.22]{Kapulkin:Homotopygroups}.
\end{proof}
It can be shown that if $X$ is connected, the homotopy groups do not depend, up to isomorphism, on the choice of the basepoint.
Given two based graphs $(X,*_X),(Y,*_Y)$ and a graph map $\vf:X\to Y$ we say the map is a \fa{based map} if $\vf(*_X)=*_Y$. To signify that a graph map has this property we write $\vf:(X,*_X)\to(Y,*_Y)$. In this setting, postcomposition with $\vf$ induces group homomorphisms
\begin{align*}
    \vf_*\from\AA_k(X,*_X) & \la \AA_k(Y,*_Y) \\
    [f]                    & \lm [\vf\circ f]
\end{align*}
\begin{definition}
    Let $X,Y$ be graphs and $\vf\from X\to Y$ a graph map.
    \begin{itemize}
        \item $\vf$ is called an $n$\fa{-equivalence} if for all $*_X\in V(X)$ the induced map $\vf_*\from \AA_k(X,*_X) \to \AA_k(Y,\vf(*_X))$ is an isomorphism for all $k< n$, and an epimorphism for $k=n$.
        \item $\vf$ is called a \fa{weak equivalence} if $\vf$ is an $n$-equivalence for all $n \geq 1$.
        \item $\vf$ is called a \fa{homotopy equivalence} if there is a map $\psi\from Y\to X$ such that $\psi\circ \vf$ and $\vf \circ \psi$ are homotopic to the identity map on $X$ and $Y$ respectively.
    \end{itemize}
\end{definition}

\begin{definition}
    A a. e. c. graph $(X,*)$ is called \fa{$n$-connected} if $\AA_k(X,*)=0$ for all $k\leq n$.
\end{definition}
\begin{lemma}\label{lem:colimithomotopygroups}
    Let $X_1,X_2,\ldots$ be a sequence of connected graphs with a family of inclusions of induced subgraphs $\iota_k: X_k \to X_{k+1}$, such that $\iota_k$ induces an isomorphism $\iota_{k*}:\AA_i(X_k) \to \AA_i(X_{k+1})$ for all $i\leq k$. Let $X_\infty$ be the colimit of this sequence. Then the universal cocone map $\hat{\iota}_k\from X_k\to X_\infty$ induces isomorphisms $\hat{\iota}_{k*}:\AA_i(X_k) \to \AA_i(X_\infty)$ for all $k\geq 1$ and $i\leq k$.
\end{lemma}
\begin{proof}
    Let $i\geq 1$. We have to show that $\hat{\iota}_k\from X_k\to X_\infty$ induces isomorphisms of homotopy groups in degrees $i\leq k$. We show that $\hat{\iota}_{k*}:\AA_i(X_k) \to \AA_i(X_\infty)$ is injective and surjective separately. Firstly, consider a a. e. c. map $f:I_\infty^i\to X_k$ such that $\hat{\iota}_{k*}([f])=0$. Then there is a a. e. c. homotopy $H:I_\infty^i\bprod I_m\to X_\infty$ from $\hat{\iota}_k\circ f$ to the constant map. Since the image of $H$ is finite, there is an $N\geq k$ such that the image of $H$ is contained in $X_N$. Writing $\iota_{N,k} \letbe \iota_{N-1} \circ \dots \circ \iota_k : X_k \to X_N$, since $(\iota_{N,k})_*:\AA_i(X_k) \to \AA_i(X_N)$ is an isomorphism, we have that $[f]=0$ in $\AA_i(X_k)$ and thus $\hat{\iota}_{k*}$ is injective. Secondly, consider a a. e. c. map $f:I_\infty^i\to X_\infty$. Since the image of $f$ is finite, there is an $N\geq k$ such that the image of $f$ is contained in $X_N$. Since $(\iota_{N,k})_*:\AA_i(X_k) \to \AA_i(X_N)$ is an isomorphism, there is a a. e. c. map $g:I_\infty^i\to X_k$ such that $(\iota_{N,k})_*([g])=[f]$ in $\AA_i(X_N)$. Thus, we have that $\hat{\iota}_{k*}([g])=[f]$ in $\AA_i(X_\infty)$ and thus $\hat{\iota}_{k*}$ is surjective.
\end{proof}

The following is Theorem 2.8 from \cite{Helene:Foundations}:
\begin{theorem}[Seifert-van Kampen theorem for graphs]\label{thm:seifertvankampen}
    Let $(X,*)$ be a based connected graph and let $A, B$ be connected, induced subgraphs of $X$, such that $A\cap B$ is connected, $*\in A\cap B$ and $A\cup B=X$. Furthermore, let the distance between $A\setminus B$ and $B\setminus A$ be at least $3$. Denote by $\iota_A,\iota_B$ the inclusions of $A\cap B$ into $A$ and $B$ respectively, and by $j_A,j_B$ the inclusions of $A$ and $B$ respectively into $X$. Then, the following diagram is a pushout in the category of groups:
    \[
        \begin{tikzcd}
            \AA_1(A\cap B,*) \ar[r,"\iota_{A*}"] \ar[d,"\iota_{B*}"] & \AA_1(A,*) \ar[d,"j_{A*}"] \\
            \AA_1(B,*) \ar[r,"j_{B*}"] & \AA_1(X,*).
        \end{tikzcd}
    \]

\end{theorem}
\subsection{Graph coverings}
Graph coverings in the sense most commonly appearing in combinatorics, namely graph maps that are local isomorphisms when restricted to $1$-neighborhoods, have been first considered in $A$-homotopy theory for graphs without triangles and squares in \cite{Hardeman:Graphcoverings}, and used in \cite{Helene:Vanishing} to show that these graphs have vanishing higher homology groups. For arbitrary graphs, Mavinkurve and Kapulkin observed in \cite{Udit:Fundamentalgroupoid} that to obtain a notion of graph covering that is more suitable for $A$-homotopy theory, one has to furthermore require a lifting property.   Consider the inclusion $\inc\from I_3\to I_1^2$, that maps $0\mapsto (0,0)$, $1\mapsto (1,0)$, $2\mapsto (1,1)$ and $3\mapsto (0,1)$.
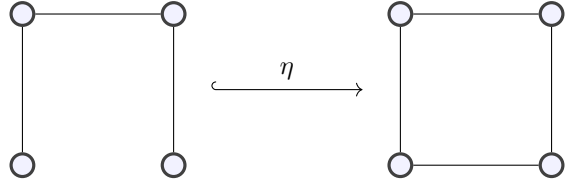
\begin{center}
    \begin{tikzpicture}

        \node[roundnode] (a1) at (0,0) {};
        \node[roundnode] (a2) at (0,2) {};
        \node[roundnode] (a3) at (2,2) {};
        \node[roundnode] (a4) at (2,0) {};

        \draw (a1) -- (a2) -- (a3) -- (a4);

        \node[roundnode] (b1) at (5,0) {};
        \node[roundnode] (b2) at (5,2) {};
        \node[roundnode] (b3) at (7,2) {};
        \node[roundnode] (b4) at (7,0) {};

        \draw (b1) -- (b2) -- (b3) -- (b4) -- (b1);

        \draw[mono] (2.5,1) -- (4.5,1) node[midway, above] {$\eta$};

    \end{tikzpicture}
\end{center}

\begin{definition}\label{dfn:cov}

    A pointed graph map $p:(X,*_X)\to (B,*_B)$ is called a \fa{covering}, if
    \begin{itemize}
        \item for all $v\in X$ the restriction of $p$ to $\Umg{1}{v}$ gives rise to an isomorphism from $\Umg{1}{v}$ to $\Umg{1}{p(v)}$,
        \item and $p$ has the right lifting property with respect to $\eta$.
    \end{itemize}
    $p$ is called a \fa{regular covering} if the image of the induced map $p_*\from \AA_1(X,*_X)\to \AA_1(B,*_B)$ is a normal subgroup.
\end{definition}
\begin{proposition}
    Let $p:(X,*_X)\to (B,*_B)$ be a covering. Then there is a right action of $\AA_1(B,*_B)$ on $p^{-1}(*_B)$ given by lifting loops in $B$ starting at $*_B$ to paths in $X$ starting at vertices in $p^{-1}(*_B)$. We call this action the \fa{monodromy action}.
\end{proposition}
The following result can be found in \cite[Proposition 4.14]{Udit:Fundamentalgroupoid}.
\begin{proposition}\label{prop:UditChrisCovering}
    For a covering $p:(X,*_X)\to (B,*_B)$, the induced map
    \[
        \begin{tikzcd}
            p_*\from\AA_n(X,*_X) \arrow{r} & \AA_n(B,*_B)
        \end{tikzcd}
    \]
    is an isomorphism of groups for $n\geq 2$ and a monomorphism for $n=1$.
\end{proposition}
\begin{definition}
    Let $p:(X,*_X)\to (B,*_B)$ be a covering. The group of deck transformations $\mathrm{Aut}(p)$ is the group of graph automorphisms $f:X\to X$ such that $p\circ f=p$.
\end{definition}

\begin{theorem}[Galois correspondence for graphs]
    Let $(X,*)$ be a connected graph, then there is an equivalence of categories
    \[
        \left\{\text{subgroups of }\AA_1(X,*) \text{ ordered by inclusion}\right\}\cong\left\{\text{connected based graph coverings of }(X,*)\right\}
    \]
    This correspondence is given by identifying a connected covering $p:(\tilde{X},*_{\tilde{X}})\to(X,*_X)$ with its image in the fundamental group $p_*(\AA_1(\tilde{X},*_{\tilde{X}}))$. The covering $p$ is regular if and only if the group of deck transformations $\mathrm{Aut}(p)$ acts transitively on $p^{-1}(*_X)$. 

\end{theorem}
\subsection{Discrete homology groups}
The following homology theory has been introduced in \cite{Helene:Homology}. To define it, consider the following families of graph maps between the cube graphs $I_1^n$:
\begin{itemize}
    \item for $1 \leq i \leq n$ the \fa{face maps} $\gd_i^\ve$
          \begin{align*}
              \gd_i^\ve\colon I_1^{n-1} & \la I_1^n                                                \\
              (v_1,\dots, v_{n-1})      & \lm (v_1,\dots,v_{i-1},\varepsilon,v_i,\dots, v_{n-1}) .
          \end{align*}
    \item the \fa{degeneracy maps} $\gs_i$, for $1 \leq i \leq n+1$:
          \begin{align*}
              \gs_i\colon I_1^{n+1} & \la I_1^{n}                                     \\
              (v_1,\dots,v_{n+1})   & \lm (v_1,\dots, v_{i-1},v_{i+1},\dots,v_{n+1}).
          \end{align*}
\end{itemize}
Now for a graph $X$ and all $n\geq 0$ denote the set of graph homomorphisms $I_1^n\to X$ by $\Graph\left(I_1^n,X\right)$.
Precomposing these with face and degeneracy maps respectively gives rise to maps of sets
\begin{align*}
    \csd_i^\ve\from\Graph\left(I_1^n,X\right) & \la \Graph\left(I_1^{n-1},X\right) \\
    f                                         & \lm f\circ \gd_i^\ve
\end{align*}
and
\begin{align*}
    \css_i\from\Graph\left(I_1^n,X\right) & \la \Graph\left(I_1^{n+1},X\right) \\
    f                                     & \lm f\circ \gs_i
\end{align*}
Applying the free Abelian group functor $\ZZ(\cdot)\from \Set\to \Ab$ we obtain a family of Abelian groups
\[
    L_nX=\ZZ\left(\Graph(I_1^n,X)\right),
\]
which can be equipped with the differential
\begin{align*}
    \del_n\from L_nX & \la L_{n-1}X                                                  \\
    f                & \lm \sum_{i=1}^{n}(-1)^{i}\left( \del_i^1f-\del_i^0f \right).
\end{align*}
It holds by a routine calculation that $\del_{n-1}\del_{n}=0$, so we obtain a chain complex $(L_*X,\del_*)$. However, this complex does not define a homology theory with desirable properties; for instance, the homology groups of the chain complex $L_*X$ for $X$ the graph with only one vertex have rank 1 in all degrees. To construct a more suitable homology one extra step is needed. For this let
\[
    D_nX=\ZZ\left( \bigcup_{i=1}^{n}\css_i\left(\Graph(I_1^{n-1},X)\right) \right),
\]
the subgroup of $L_nX$ generated by the images of the degeneracy maps. It can easily be seen that $\del(D_nX)\subseteq D_{n-1}X$ and thus, the family $D_*X$ forms a chain sub-complex of $L_*X$.
\begin{definition}
    The \fa{discrete homology} of a graph $X$ is the homology of the chain complex
    \[
        C_*(X)\letbe L_*X/D_*X.
    \]
    For every $n\geq 0$ we write $H_n(X)$ for $H_n(C_*(X))$.
\end{definition}
A generalized variant of the following theorem can be found in \cite[Thm 3.4]{Helene:Homology}; in \cite{Helene:Vanishing} a specialization to graphs can be found.
\begin{theorem}[Mayer-Vietoris sequence for graphs]
    Let $X$ be a graph, $n\in \NN$, and let $A,B$ be two induced subgraphs of $X$ such that $X=A\cup B$ and the distance between $A\setminus B$ and $B\setminus A$ is at least $n+2$, then there is an exact sequence in homology in degrees $\leq n$:
    \[
        \begin{tikzcd}
            H_n(A\cap B) \ar[r,"\iota_*"] &H_n(A)\oplus H_n(B)\ar[r,"j_*"] \ar[d, phantom, ""{coordinate, name=Z}] &H_n(X) \ar[dll,swap,"\delta", rounded corners,
                to path={ -- ([xshift=8.5ex]\tikztostart.center)
                        |- (Z) [near end]\tikztonodes
                        -| ([xshift=-13ex]\tikztotarget.center) -- (\tikztotarget)}] \\
            H_{n-1}(A\cap B) \ar[r,"\iota_*"] &H_{n-1}(A)\oplus H_{n-1}(B)\ar[r,"j_*"] \ar[d, phantom, ""{coordinate, name=W}] &\cdots \ar[dll, rounded corners,
                to path={ -- ([xshift=8.5ex]\tikztostart.center)
                        |- (W) [near end]\tikztonodes
                        -| ([xshift=-13ex]\tikztotarget.center) -- (\tikztotarget)}] \\
            H_0(A\cap B) \ar[r,"\iota_*"] &
            H_0(A)\oplus H_0(B)\ar[r,"j_*"] &
            H_0(X),
        \end{tikzcd}
    \]
    where $\iota(x)=(x,x)$ and $j(x,y)=x-y$.
\end{theorem}

\begin{definition}
    Let $X$ be a graph, $A$ a subgraph of $X$ and let $\iota:A\to X$ be the inclusion. The \fa{relative chain complex} $C_*(X,A)$ of $(X,A)$ is the cokernel of the inclusion $\iota_*:C_*(A) \to C_*(X)$.
    We define the \fa{relative homology groups} of $(X,A)$, denoted $H_n(X,A)$, to be the homology of the chain complex $C_*(X,A)$.
\end{definition}
The short exact sequence of chain complexes
\[
    0  \to C_*(A) \xrightarrow{\iota_*} C_*(X) \to C_*(X,A) \to 0
\]
gives rise to a long exact sequence in homology.

\subsection{Relative homotopy groups}
Relative homotopy groups were defined in \cite{Helene:Foundations}. Given $n\geq 1$ and a pair $(X,A)$, where $X$ is a graph and $A$ is a subgraph of $X$ and a vertex $*\in V(A)$, consider the sets of graph maps
\[
    \Graph((I_m^n,\del I_m^n,\hg{m}{n}),(X,A,*)),
\]
i.e. the set of graph maps
\begin{align*}
    f:(I_m^n,\del I_m^n,\hg{m}{n}) \to (X,A,*).
\end{align*}
Now consider the equivalence relation $\sim$ generated by the suitable notion of homotopy, that is maps
\begin{align*}
    H:(I_m^n\bprod I_k,\del I_m^n\bprod I_k,\hg{m}{n}\bprod I_k) \to (X,A,*).
\end{align*}
Every map $\phi:I_M^n\to I_m^n$ in $\cat{LR}_n$ satisfies $\phi(\hg{M}{n})= \hg{m}{n}$ and $\phi(\del I_M^n)=\del I_m^n$ and hence induces a map
\[
    \phi^*: \Graph((I_{m}^n,\del I_{m}^n,\hg{m}{n}),(X,A,*)) \to \Graph((I_{M}^n,\del I_{M}^n,\hg{M}{n}),(X,A,*)).
\]
This map is compatible with homotopies, in the sense that homotopic maps get sent to homotopic maps.
Now let
\[
    \AA_n(X,A,*)\letbe \left(\bigsqcup_{m\in\NN}\Graph((I_m^n,\del I_m^n,\hg{m}{n}),(X,A,*))\right) /\sim.
\]
where $\sim$ is the equivalence relation generated by the homotopies and precomposition with morphisms in $\cat{LR}_n$.
Every map of pairs $\vf:(X,A)\to (Y,B)$ gives rise to a map $\vf_*:\AA_*(X,A,*)\to\AA_*(Y,B,*)$, which is a group homomorphism in every degree $n\geq 2$ and a map of pointed sets in degree $n=1$.
Also, there is a long exact sequence of homotopy groups
\[
    \begin{tikzcd}
        \cdots \ar[r] &
        \AA_n(X,*)
        \ar[r] \ar[d, phantom, ""{coordinate, name=Z}] &
        \AA_n(X,A,*) \ar[dll,swap,"(\gd_n^1)^*", rounded corners,
            to path={ -- ([xshift=8.5ex]\tikztostart.center)
                    |- (Z) [near end]\tikztonodes
                    -| ([xshift=-13ex]\tikztotarget.center) -- (\tikztotarget)}] \\
        \AA_{n-1}(A,*) \ar[r] &
        \AA_{n-1}(X,*) \ar[r] \ar[d, phantom, ""{coordinate, name=W}] &
        \cdots \ar[dll, rounded corners,
            to path={ -- ([xshift=8.5ex]\tikztostart.center)
                    |- (W) [near end]\tikztonodes
                    -| ([xshift=-13ex]\tikztotarget.center) -- (\tikztotarget)}] \\
        \AA_0(A,*) \ar[r] &
        \AA_0(X,*)
    \end{tikzcd}
\]
where the boundary map $(\gd_n^1 )^*:\AA_{n}(X,A,*)\to \AA_{n-1}(A,*)$ is given as follows: Let $f\in \AA_{n}(X,A,*)$ be represented by $f:(I_m^n,\del I_m^n,\hg{m}{n})\to (X,A,*)$, then the image of the composition
\[
    \begin{tikzcd}
        (I_m^{n-1},I_m^{n-1},\del I_m^{n-1})\ar[r,"\gd_n^1"] &(I_m^n,\del I_m^n,\hg{m}{n})\ar[r,"f"] & (X,A,*)
    \end{tikzcd}
\]
is contained in $(A)$. Hence, we obtain a map $(\gd_n^1)^*f:(I_m^{n-1},\del I_m^{n-1})\to (A,*)$.

\subsection{Cubical sets}
In this subsection we collect some definitions and results needed later and refer for more details to \cite{Chris:Revisited} and \cite{Kapulkin:Homotopygroups}.
\begin{definition}
    The \fa{cube category} $\square$ (with \fa{connections}) is the category with objects $[1]^n$, $n \geq 0$, and morphisms generated by:
    \begin{enumerate}
        \item \fa{face maps} $\cd_i^\ve$ for $1\leq i \leq n$
              \begin{align*}
                  \cd_i^\ve\from \one^{n-1} & \la \one^n                                              \\
                  (v_1,\ldots, v_{n-1})     & \lm (v_1,\ldots,v_{i-1},\varepsilon,v_i,\dots, v_{n-1})
              \end{align*}
        \item \fa{degeneracies} $\cs_i$ for $1\leq i \leq n+1$
              \begin{align*}
                  \cs_i\from \one^{n+1} & \la \one^{n}                                     \\
                  (v_1,\ldots,v_{n+1})  & \lm (v_1,\ldots ,v_{i-1},v_{i+1},\ldots,v_{n+1})
              \end{align*}
        \item \fa{connections} $\cg_i^\ve$ for $1\leq i \leq n$
              \begin{align*}
                  \cg_i^\ve\from \one^{n+1} & \la \one^{n}                                               \\
                  (v_1,\ldots, v_{n+1})     & \lm (v_1,\ldots,v_{i-1},m^\ve(v_i,v_{i+1}),\ldots v_{n+1})
              \end{align*}
              where $m^\ve$ is the maximum (minimum) for $\ve=0$ ($\ve=1$).
    \end{enumerate}
    A \fa{cubical set} is a presheaf on the cube category $\square$. The \fa{category of cubical sets} is denoted by $\cSet$ and the morphisms in $\cSet$ are the natural transformations. For a cubical set $\mathsf{X}$, we write $\csd_i^\ve$, $\css_i$ and $\csg_i^\ve$ for the maps $\mathsf{X}(\cd_i^\ve)$, $\mathsf{X}(\cs_i)$ and $\mathsf{X}(\cg_i^\ve)$ respectively.
\end{definition}
From a cubical set $\mathsf{X}$ a CW-complex $|\mathsf{X}|$ can be obtained by taking the \fa{geometric realization} of $\mathsf{X}$. This is the space
\[
    |\mathsf{X}|\letbe \left( \bigsqcup_{n\in\NN} [0,1]^n\times \mathsf{X}(\one^n) \right)/\sim,
\]
where the equivalence relation $\sim$ is generated by
\begin{align*}
    \left(\,(t_1,\ldots ,t_{n-1}), \csd_i^\ve x\right) & \sim \left(\,(t_1,\ldots, t_{i-1}, \ve, t_{i},\ldots, t_{n-1}), x\right),                  \\
    \left(\,(t_1,\ldots ,t_{n+1}), \css_i x\right)     & \sim \left(\,(t_1,\ldots, t_{i-1}, t_{i+1},\ldots, t_{n+1}), x\right),                     \\
    \left(\,(t_1,\ldots ,t_{n+1}), \csg_i^\ve x\right) & \sim \left(\,(t_1,\ldots, t_{i-1}, m^\ve(t_i,t_{i+1}),t_{i+2}, \ldots, t_{n+1}), x\right).
\end{align*}
This construction is functorial; given a map of cubical sets $f:\mathsf{X}\to \mathsf{Y}$ and an element $((t_1,\ldots, t_n),x)\in |\mathsf{X}|$, we define $|f|((t_1,\ldots, t_n),x)=((t_1,\ldots,t_n),f(x))$. The \fa{skeleton} functor $\sk_n:\cSet\to\cSet$ is given as follows: Given a cubical set $\mathsf{X}$, let $\sk_n(\mathsf{X})(\one^m)=\mathsf{X}(\one^m)$ for $m\leq n$ and for $m>n$ let $\sk_n(\mathsf{X})(\one^m)$ be the set of those elements in $\mathsf{X}(\one^m)$ which can be obtained via a sequence of connections and degeneracies from $\sk_{n}(\mathsf{X})(\one^n)$. The structure maps $\csd_i^\ve$, $\css_i$ and $\csg_i^\ve$ are given as the restrictions of those of $\mathsf{X}$. Observe that $|\sk_n(\mathsf{X})|$ is the $n$-skeleton, in the sense of CW-complexes, of $|\mathsf{X}|$.

Our main motivation to study cubical sets is the following family of cubical sets $N_m X, m\in \NN_{>0}$ obtained from a graph $X$:
We define the cubical set $N_m X$ by setting
\begin{align*}
    N_m X(\one^n) = \mathrm{Hom}_{\Graph}(I_m^n, X)
\end{align*}
and maps given by precomposing $f \in \mathrm{Hom}_{\Graph}(I_m^n, X)$
with the graph maps
\begin{align*}
    \gd_i^\ve\from  I_m^{n-1} & \la I_m^n                                               \\
    (v_1,\dots, v_{n-1})      & \lm (v_1,\dots,v_{i-1},\ve\cdot m, v_i, \dots, v_{n-1})
\end{align*}
for $1 \leq i \leq n$ to yield the face maps $\csd_i^\ve$, with the graph maps
\begin{align*}
    \gs_i\from  I_m^{n+1} & \la I_m^{n}                                    \\
    (v_1,\dots,v_{n+1})   & \lm (v_1,\dots, v_{i-1},v_{i+1},\dots,v_{n+1})
\end{align*}
for $1 \leq i \leq n+1$ to yield the degeneracy maps $\css_i$ and with the graph maps
\begin{align*}
    \gg_i^\ve\from  I_m^{n+1} & \la I_m^{n}                                               \\
    (v_1,\dots, v_{n+1})      & \lm (v_1,\dots,v_{i-1},m^\ve(v_i,v_{i+1}),\dots, v_{n+1})
\end{align*}
to yield the connections $\csg_i^\ve$.

For $m\geq 1$ recall the graph maps
\begin{align*}
    \grefl_m: I_{m+1} & \la I_{m}       &  & \text{and  } & \grefr_m:I_{m+1} & \la I_m        \\
    t                 & \lm \max(0,t-1) &  &              & t                & \lm \min(m,t).
\end{align*}
and let
\begin{align*}
    \gref_m: I_{m+2} & \la I_m                  \\
    t                & \lm \max(0,\min(t-1,m)).
\end{align*}
Note that $\gref_m=\grefl_{m}\circ\grefr_{m+1}=\grefr_{m}\circ\grefl_{m+1}$.
Precomposition with $\grefl_m^n,\grefr_m^n: I_{m+1}^n \to I_m^n$ gives rise to maps of cubical sets $\crefl_m,\crefr_m:N_{m}(X)\to N_{m+1}(X)$. This allows us to consider the colimit
\begin{align*}
    NX & \letbe \colim{}\left( N_1 X\xrightarrow{\crefl_1} N_2 X\xrightarrow{\crefr_2} N_3 X\xrightarrow{\crefl_3} \cdots \right),
\end{align*}
where the sequence alternates between $\crefl$ and $\crefr$.
Denote by $\cref: N_1 X \to NX$ the canonical map.
\begin{theorem}\label{magic}
    For every $k\in \NN_{>0}$, there is, up to natural isomorphism, the following commutative diagram of functors
    \[
        \begin{tikzcd}
            \Graph \ar[r,"N"] \ar[d,"\AA_k"] & \cSet \ar[d,"| \cdot | "] \\
            \Grp & \ar[l,"\pi_k"] \Top
        \end{tikzcd}
    \]
    where $\AA_k:\Graph \to \Grp$ is the $k$-th discrete homotopy group and $\pi_k:\Top\to\Grp$ is the $k$-th homotopy group of a topological space; for $k=0$ both $\AA_0$ and $\pi_0$ are functors into pointed sets, the corresponding diagram also commutes. Furthermore, the maps $|\crefl_m|,|\crefr_m|:|N_m X|\to |N_{m+1} X|$ and $\tref:|N_1 X|\to |NX|$ are weak equivalences.
\end{theorem}
The natural transformation $\AA_k(X) \to \pi_k(|NX|)$ is given as follows: Let $f:(I_m^k,\del I_m^k) \to (X,*)$ be a representative of a homotopy class in $\AA_k(X,*)$. Then $f$ is by definition of $N_m$ an element of $N_m X(\one^k)$ and can hence be identified with an element of $NX(\one^k)$ via the canonical map $N_m X\to NX$. Now the following map
\begin{align*}
    |\bar{f}|:([0,1]^k,\del[0,1]^k) & \la (|NX|,*)          \\
    (t_1,...,t_k)                   & \lm ((t_1,...,t_k),f)
\end{align*}
is a representative of a homotopy class of $(|NX|, *)$.

In homology, we have a similar result to \ref{magic}: Observe that there is exactly one $n$-cell in $|N_1 X|$ for every graph map $f:I_1^n\to X$, that is neither a connection nor a degeneracy. Using that the subcomplex generated by all connections is acyclic (\cite{Welker:Connections}), we obtain an isomorphism
\begin{align*}
    H_*(X) & \cong H_*(|N_1 X|).
\end{align*}
Similarly, for a pair of graphs $A\subseteq X$ it holds that $H_*(X,A)\cong H_*(|N_1 X|,|N_1 A|)$.
\subsection{Fibrations in discrete homotopy theory}
This entire subsection is based on section 5 in \cite{Chris:Revisited}.
\begin{definition}
    A graph map $\vf:X\to Y$ is called a \fa{fibration}, if for every $n\in\NN$ and every diagram of the form
    \[
        \begin{tikzcd}
            \hg{m}{n}\ar[r,"\hat{g}"]\ar[d,hook]&X\ar[d,"\vf"]\\
            I_m^{n}\ar[r,"g"]&Y
        \end{tikzcd}
    \]
    there is a $k\in\NN$ and a map $h:I_{m+2k}^{n}\to X$ such that the following diagram commutes
    \[
        \begin{tikzcd}
            \hg{m+2k}{n}\ar[r,"c"]\ar[d,hook]  &\hg{m}{n}\ar[r,"\hat{g}"]&X\ar[d,"\vf"]\\
            I_{m+2k}^{n}\ar[r,"c"]\ar[rru,dashed,"h"]&I_m^{n}\ar[r,"g"]&Y,
        \end{tikzcd}
    \]
    where $c:I_{m+2k}^{n}\to I_m^{n}$ is the $n$-fold box product of the map $$I_{m+2k}\to I_m,t\mapsto\max(0,\min(t-k,m)).$$
\end{definition}

\begin{theorem}[\cite{Chris:Revisited}, Theorem~5.14]\label{thm:LESFibration}
    Let $\vf:(X,*_X)\to(Y,*_Y)$ be a based fibration of graphs and let $F=\vf^{-1}(*_Y)$. Then there is a long exact sequence
    \[
        \begin{tikzcd}
            \cdots \ar[r] &
            \AA_n(X,*_X)
            \ar[r] \ar[d, phantom, ""{coordinate, name=Z}] &
            \AA_n(Y,*_Y) \ar[dll,swap, rounded corners,
                to path={ -- ([xshift=8.5ex]\tikztostart.center)
                        |- (Z) [near end]\tikztonodes
                        -| ([xshift=-10ex]\tikztotarget.center) -- (\tikztotarget)}] \\
            \AA_{n-1}(F, *_X) \ar[r] &
            \AA_{n-1}(X,*_X) \ar[r] \ar[d, phantom, ""{coordinate, name=W}] &
            \cdots \ar[dll, rounded corners,
                to path={ -- ([xshift=8.5ex]\tikztostart.center)
                        |- (W) [near end]\tikztonodes
                        -| ([xshift=-10ex]\tikztotarget.center) -- (\tikztotarget)}] \\
            \AA_0(F,*_X) \ar[r] &
            \AA_0(X,*_X) \ar[r] & \AA_0(Y,*_Y)
        \end{tikzcd}
    \]

\end{theorem}
\begin{definition}\label{dfn:fibrationcategory}
    A fibration category is a category $\cat{C}$ with two subcategories $\cat{F}$ and $\cat{W}$ satisfying the following conditions. We call the morphisms of $\cat{F}$ fibrations and the morphisms of $\cat{W}$ weak equivalences. A map that is a morphism in both is called an acyclic fibration.
    \begin{enumerate}
        \item Given two morphisms:
              \[ X \xrightarrow{f} Y \xrightarrow{g} Z \]
              if two of $f$, $g$, $gf$ are weak equivalences then so is the third;
        \item all isomorphisms are acyclic fibrations;
        \item pullbacks along fibrations exist; fibrations and acyclic fibrations are stable under pullback;
        \item $\cat{C}$ has a terminal object $1$; the map $X \to 1$ is a fibration for any object $X \in \cat{C}$;
        \item every map can be factored as a weak equivalence followed by a fibration.
    \end{enumerate}
\end{definition}
The following is Theorem 5.9 in \cite{Chris:Revisited}.
\begin{theorem}
    The category $\Graph$ with weak equivalences and fibrations as above is a fibration category.
\end{theorem}
We give a brief overview of the construction behind the factorization of a map into a weak equivalence and a fibration, which relies on the notion of a path graph. The path graph has first been defined in \cite{Babson:CubicalApprox}, however, we stick to the slightly different notion defined in \cite{Chris:Revisited}.
\begin{definition}
    For a graph $Y$, let $P_m Y$ be the graph with vertices $V(P_mY)=\hom_{\Graph}(I_m,Y)$, in which two distinct vertices $f,g$ are connected if and only if for all $v\in V(I_m)$ it holds that $\{f(v),g(v)\}\in E(Y)$ or $f(v)=g(v)$. The \fa{path graph of }$Y$ is the colimit $PY$ of
    \begin{align*}
        P_1Y\hookrightarrow P_2Y\hookrightarrow \ldots,
    \end{align*}
    where the map $P_{m-1}Y\hookrightarrow P_mY$ maps $\omega:I_{m-1}\to Y$, to the path
    \begin{align*}\label{diag:mappingpathgraph}
        I_{m} & \la Y \\
        v     & \lm
        \begin{cases}
            \omega(v),   & v<m  \\
            \omega(m-1), & v=m.
        \end{cases}
    \end{align*}
\end{definition}
With this, there is a way to construct two acyclic fibrations $e_0,e_1:PY\to Y$, where $e_0$ maps every path to its starting point and $e_1$ maps every path to its endpoint. Furthermore, the induced map $(e_0,e_1):PY\to Y\times Y$
is a fibration.

Now, for $f:(X,*_X)\to (Y,*_Y)$, consider the following diagram, where $E_f$, the \fa{mapping path graph}, is defined as the pullback of the square
\[
    \begin{tikzcd}
        E_f\ar[r]\ar[d]& X\ar[d,"f"]\\
        PY\ar[r,"e_1"]\ar[d,"e_0"]& Y\\
        Y
    \end{tikzcd}
\]
It can be shown that the vertical map $p_f:E_f\to Y$ is a fibration and the horizontal map $E_f\to X$ is an acyclic fibration. The \fa{homotopy fiber of $f$}, denoted by $\mathrm{hofib}(f)$, is given by $p_f^{-1}(*_Y)$. Choosing a basepoint $*\in \mathrm{hofib}(f)$, we then get a long exact sequence
\[
    \begin{tikzcd}
        \cdots \ar[r] &
        \AA_n(X,*_X)
        \ar[r] \ar[d, phantom, ""{coordinate, name=Z}] &
        \AA_n(Y,*_Y) \ar[dll,swap, rounded corners,
            to path={ -- ([xshift=8.5ex]\tikztostart.center)
                    |- (Z) [near end]\tikztonodes
                    -| ([xshift=-10ex]\tikztotarget.center) -- (\tikztotarget)}] \\
        \AA_{n-1}(\mathrm{hofib}(f),*) \ar[r] &
        \AA_{n-1}(X,*_X) \ar[r] \ar[d, phantom, ""{coordinate, name=W}] &
        \cdots \ar[dll, rounded corners,
            to path={ -- ([xshift=8.5ex]\tikztostart.center)
                    |- (W) [near end]\tikztonodes
                    -| ([xshift=-10ex]\tikztotarget.center) -- (\tikztotarget)}] \\
        \AA_0(\mathrm{hofib}(f),*) \ar[r] &
        \AA_0(X,*_X) \ar[r] &
        \AA_0(Y,*_Y).
    \end{tikzcd}
\]
Note that we can give a concrete description of $E_f$ and $\mathrm{hofib}(f)$ by using how limits and colimits are computed in $\Graph$. Every path $\omega\in PY$ is eventually constant, we will denote the endpoint of $\omega$ by $\lim{k\to \infty} \omega(k)$.
\begin{align*}
    E_f               & = \left\{\,(\omega,x)\in PY\times X \,\middle|\, \lim{k\to \infty} \omega(k) = f(x)\,\right\}                \\
    \mathrm{hofib}(f) & =\left\{\,(\omega,x)\in PY\times X \,\middle|\,\lim{k\to \infty}\omega(k) = f(x),\, \omega(0)=*_Y \,\right\}
\end{align*}
Furthermore, the upper horizontal map can be characterized as
\begin{align*}
    \pr{2}:E_f & \la X \\
    (\omega,x) & \lm x
\end{align*}
which admits a right inverse $s: X \to E_f$ defined by $s(x) = (c_{f(x)}, x)$, where $c_{f(x)}$ denotes the constant path at $f(x)$. Since $s$ is a section of the projection $\pr{2}$, and the map $\pr{2}$ is a weak equivalence, so is $s$. Hence we obtain a factorization of $f$ as a weak equivalence followed by a fibration as follows
\begin{align*}
    X \overset{s}{\la} E_f \overset{p_f}{\la} Y.
\end{align*}

\section{A homotopy inverse}
In this section, we construct a homotopy inverse to the canonical map $|N_1X|\to |NX|$. It has been shown in \cite{Chris:Revisited} that this map is a weak equivalence, therefore, since both $|N_1X|$ and $|NX|$ are CW-complexes, the existence of a homotopy inverse follows from Whitehead's theorem, and is therefore not surprising. However, this concrete description will allow us to also construct a relative map $\AA_*(X,A,*)\to \pi_*(|N_1X|,|N_1A|,*)$, which will be shown to be an isomorphism using the five lemma. In particular, this will follow a relative version of the Hurewicz theorem.
\subsection{Construction of the inverse}
Consider the following two auxiliary maps for $m=2k+1, k\in\NN$:
\begin{align*}
    \hidx_m:[0,1] & \la \{0,1,\ldots ,m-1\}                                                                                                                      \\
    t             & \lm \begin{cases}
                            0   & \for t< \frac{1}{2^{k+1}}                                                            \\
                            k-l & \for \frac{1}{2^{l+2}}\leq t<\frac{1}{2^{l+1}} \text{ and } l\in\{1,\ldots,k-1\}     \\
                            k   & \for \frac{1}{4} \leq t< \frac{3}{4}                                                 \\
                            k+l & \for 1-\frac{1}{2^{l+1}}\leq t<1-\frac{1}{2^{l+2}} \text{ and } l\in\{1,\ldots,k-1\} \\
                            m-1 & \for t\geq 1-\frac{1}{2^{k+1}}
                        \end{cases}
\end{align*}
\begin{figure}[htbp]
    \begin{tikzpicture}[xscale=1.2,yscale=1.2]
        \draw (0,0) -- (12,0) node[right, font=\small] {$t$};
        \draw (0,0) -- (0,6) node[above,yshift=2pt, font=\small] {$\hidx_7(t)$};
        \foreach \y in {0,1,2,3,4,5,6}{
                \draw (-0.08,\y) -- (0.08,\y);
                \node[left, xshift=-2pt, font=\small] at (-0.08,\y) {\y};
            }
        \draw[darkgray,dotted] (0.75,0)  -- (0.75,1);
        \draw[darkgray,dotted] (1.5,0)   -- (1.5,2);
        \draw[darkgray,dotted] (3,0)     -- (3,3);
        \draw[darkgray,dotted] (9,0)     -- (9,4);
        \draw[darkgray,dotted] (10.5,0)  -- (10.5,5);
        \draw[darkgray,dotted] (11.25,0) -- (11.25,6);

        \foreach \x/\lab in {0/{$0$}, 0.75/{$\frac{1}{16}$}, 1.5/{$\frac{1}{8}$},
        3/{$\frac{1}{4}$}, 9/{$\frac{3}{4}$},
        10.5/{$\frac{7}{8}$}, 11.25/{$\frac{15}{16}$}, 12/{$1$}}{
        \draw (\x,-0.08) -- (\x,0.08);
        \node[below, yshift=-2pt] at (\x,-0.08) {\lab};
        }
        \draw[very thick, Violet] (0,0)     -- (0.75,0);
        \draw[very thick, Violet] (0.75,1)  -- (1.5,1);
        \draw[very thick, Violet] (1.5,2)   -- (3,2);
        \draw[very thick, Violet] (3,3)     -- (9,3);
        \draw[very thick, Violet] (9,4)     -- (10.5,4);
        \draw[very thick, Violet] (10.5,5)  -- (11.25,5);
        \draw[very thick, Violet] (11.25,6) -- (12,6);

        \foreach \p in {(0,0), (0.75,1), (1.5,2), (3,3), (9,4), (10.5,5), (11.25,6), (12,6)}{
                \filldraw[Violet] \p circle (2pt);
            }
        \foreach \p in {(0.75,0), (1.5,1), (3,2), (9,3), (10.5,4), (11.25,5)}{
                \draw[Violet, fill=white] \p circle (2pt);
            }

    \end{tikzpicture}
\end{figure}
and
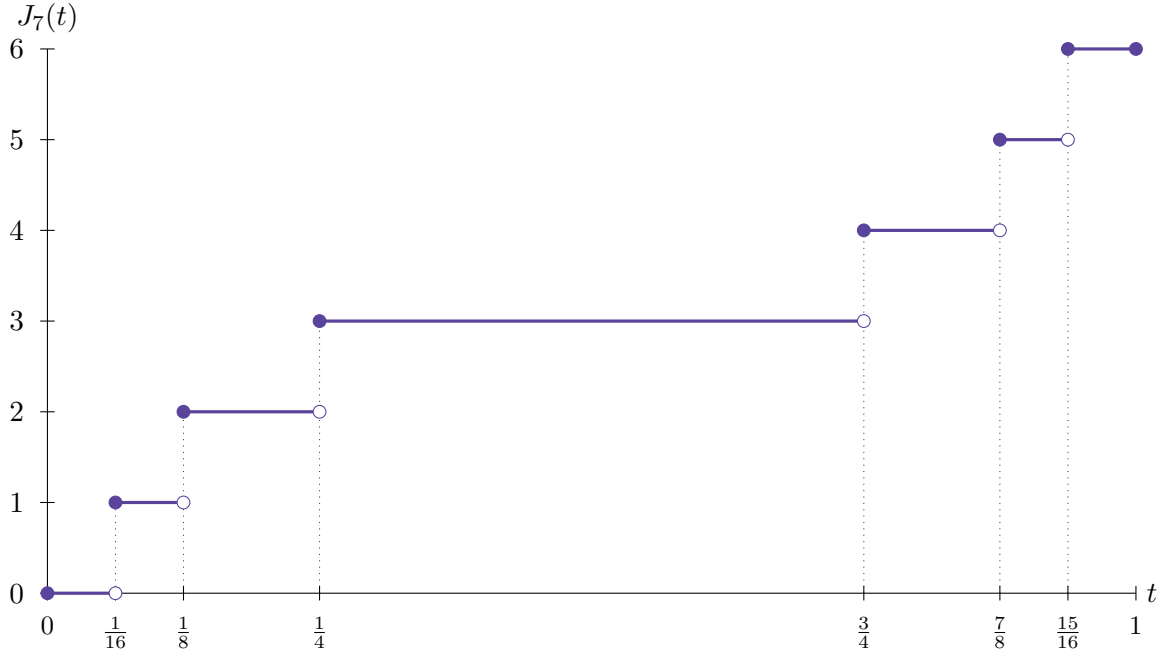
\begin{align*}
    \hcoord_m:[0,1] & \la [0,1]                                                                                                                                                                                   \\
    t               & \lm \begin{cases}
                              0                                                                                 & \for t < \frac{1}{2^{k+2}}                    \\
                              2^{k-\hidx_m(t)+2}\cdot\left(t-\frac{1}{2^{k-\hidx_m(t)+2}}\right)                & \for \frac{1}{2^{k+2}} \leq t < \frac{1}{4}   \\
                              2\cdot\left(t-\frac{1}{4}\right)                                                  & \for \frac{1}{4} \leq t < \frac{3}{4}         \\
                              2^{\hidx_m(t)-k+2}\cdot\left(t-\left(1-\frac{1}{2^{\hidx_m(t)-k+1}}\right)\right) & \for \frac{3}{4} \leq t < 1-\frac{1}{2^{k+2}} \\
                              1                                                                                 & \for t \geq 1-\frac{1}{2^{k+2}}
                          \end{cases}
\end{align*}

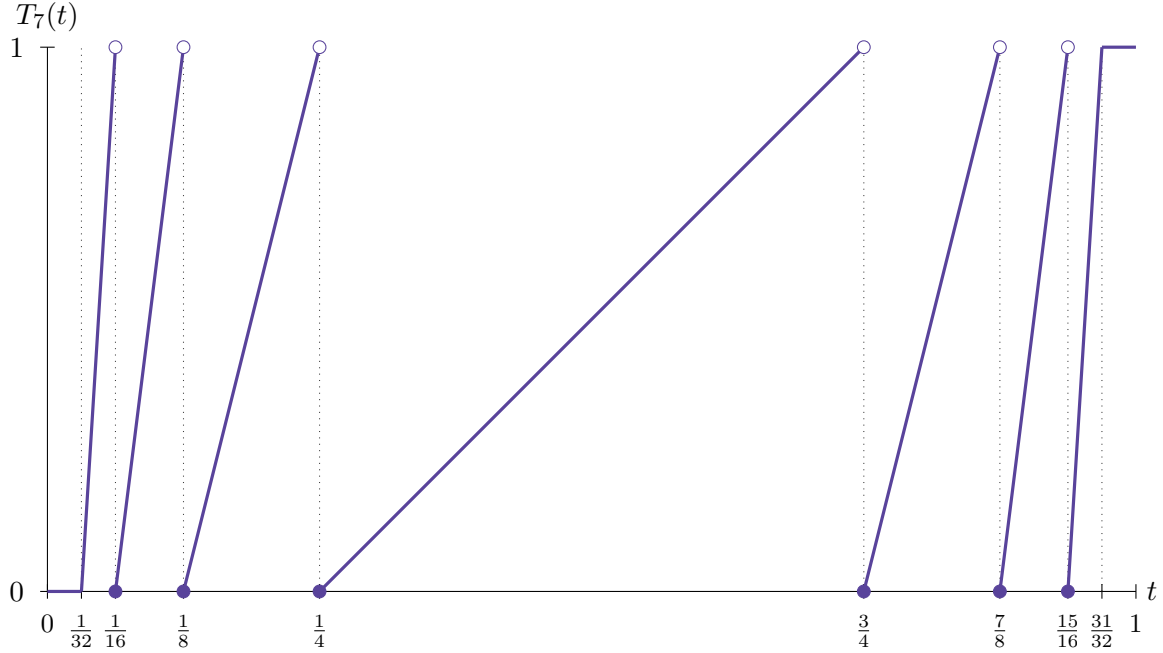
\begin{figure}[htbp]
    \begin{tikzpicture}[xscale=1.2,yscale=1.2]
        \draw (0,0) -- (12,0) node[right, font=\small] {$t$};
        \draw (0,0) -- (0,6) node[above,yshift=2pt, font=\small] {$\hcoord_7(t)$};
        \draw (-0.08,0) -- (0.08,0);
        \node[left, xshift=-2pt, font=\small] at (-0.08,0) {$0$};
        \draw (-0.08,6) -- (0.08,6);
        \node[left, xshift=-2pt, font=\small] at (-0.08,6) {$1$};

        \foreach \x in {0.375,0.75,1.5,3,9,10.5,11.25,11.625}{
                \draw[darkgray,dotted] (\x,0) -- (\x,6);
            }

        \foreach \x/\lab in {0/{$0$}, 0.375/{$\frac{1}{32}$}, 0.75/{$\frac{1}{16}$},
        1.5/{$\frac{1}{8}$}, 3/{$\frac{1}{4}$}, 9/{$\frac{3}{4}$},
        10.5/{$\frac{7}{8}$}, 11.25/{$\frac{15}{16}$},
        11.625/{$\frac{31}{32}$}, 12/{$1$}}{
        \draw (\x,-0.08) -- (\x,0.08);
        \node[below, yshift=-2pt] at (\x,-0.08) {\lab};
        }
        \draw[very thick, Violet] (0,0)        -- (0.375,0);
        \draw[very thick, Violet] (0.375,0)    -- (0.75,6);
        \draw[very thick, Violet] (0.75,0)     -- (1.5,6);
        \draw[very thick, Violet] (1.5,0)      -- (3,6);
        \draw[very thick, Violet] (3,0)        -- (9,6);
        \draw[very thick, Violet] (9,0)        -- (10.5,6);
        \draw[very thick, Violet] (10.5,0)     -- (11.25,6);
        \draw[very thick, Violet] (11.25,0)    -- (11.625,6);
        \draw[very thick, Violet] (11.625,6)   -- (12,6);

        \foreach \p in {(0.75,0), (1.5,0), (3,0), (9,0),
                (10.5,0), (11.25,0)}{
                \filldraw[Violet] \p circle (2pt);
            }
        \foreach \p in {(0.75,6), (1.5,6), (3,6), (9,6), (10.5,6), (11.25,6)}{
                \draw[Violet, fill=white] \p circle (2pt);
            }
    \end{tikzpicture}
\end{figure}
Of course, neither $\hidx_m$ nor $\hcoord_m$ is continuous; however, the discontinuity points of $\hidx_m$ are the same as those of $\hcoord_m$, it even holds that $\hidx_m+\hcoord_m$ is continuous. Given fixed $m,n\in \NN$ and a $\mathbf v=(v_1,\ldots,v_n)\in \{0,1,\ldots, m-1\}^n$, denote by $\iota_{\mathbf{v}}:I_1^n\to I_m^n$ the map $\iota_{\mathbf{v}}(t_1,\ldots,t_n)=(v_1+t_1,v_2+t_2,\ldots, v_n+t_n)$. This map $\iota_{\mathbf{v}}$ is the inclusion of a small ``side-length one'' cube into the larger ``side-length $m$'' cube, and in particular a graph map. For $t=(t_1,\ldots, t_n)\in[0,1]^n$ we denote by
\[\mathbf\hidx^n_m(t)=\left(\hidx_m(t_1),\ldots,\hidx_m(t_n)\right)\text{ and }\mathbf\hcoord^n_m(t)=\left(\hcoord_m(t_1),\ldots,\hcoord_m(t_n)\right).
\]
Given a graph $X$ and a graph map $f:I_m^n\to X$, with $m$ odd, we construct a map $\chop_f$ from  $[0,1]^n$ to $|N_1 X|$ as follows:
\begin{align*}
    \chop_f\from[0,1]^n & \la|N_1 X|                                                       \\
    t                   & \lm (\mathbf\hcoord^n_m(t),f\circ \iota_{\mathbf{\hidx}^n_m(t)})
\end{align*}
\begin{proposition}
    The map
    \begin{align*}
        \chop_m\from|N_m X| & \la|N_1 X|     \\
        (t,f)               & \lm \chop_f(t)
    \end{align*}
    is well defined and continuous for every odd $m\in\NN_{\geq 1}$; note that for $m=1$, i.e. $k=0$, the second and the fourth cases in the definitions of $\hidx_m$ and $\hcoord_m$ have empty range.
\end{proposition}
For $m>0$ even, we define $\chop_m\letbe \chop_{m+1}\circ |\crefr_m|$.
\begin{proof}
    We start by showing that $\chop_f$ is continuous for a fixed graph map $f\from I_m^n\to X$ with $m=2k+1$. Let
    $$
        p_0=0,~p_1=\frac{1}{2^{k+1}},~ \ldots , p_k=\frac{1}{4},~ p_{k+1}=\frac{3}{4},~\ldots , p_{m-1} = 1-\frac{1}{2^{k+1}},~ p_m=1
    $$
    be the discontinuity points and boundary points of the domain of $\hidx_m$ and $\hcoord_m$. We will show the continuity of $\chop_f$ via sequential continuity, for this assume we are given a convergent sequence $(x^{(j)}_1,x^{(j)}_2,\ldots,x^{(j)}_n)_{j\in\NN}$ in $[0,1]^n$ with limit $x=(x_1,\ldots,x_n)$. We will show
    \[
        \lim{j\to\infty} \chop_f(x^{(j)}) = \chop_f(x).
    \]
    The restriction of $\chop_f$ to each half-open cube $W=[p_{i_1},p_{i_1+1})\times \cdots \times [p_{i_n},p_{i_n+1})$ is continuous, since restricted to each of these subcubes $\hidx_m$ is constant and $\hcoord_m$ is continuous. Note that the boundary of each of these half-open intervals can be chosen to be closed if $p_{i_j+1}=1$. Therefore, the only nontrivial case is when there is at least one coordinate of $x$ which is a discontinuity point. Let $W=[p_{i_1},p_{i_1+1})\times \cdots \times [p_{i_n},p_{i_n+1})$ be a subcube containing a convergent subsequence $(y^{(j)})=(y_1^{(j)},\ldots,y_n^{(j)})$ of $x^{(j)}$, then $x$ has to lie in the closure of this subcube. For simplicity, we may assume, up to permutation of the entries, that
    $$x=(p_{i_1+1},\ldots p_{i_l+1},p_{i_{l+1}},\ldots p_{i_{l+s}},x_{l+s+1},\ldots x_{n}),$$
    with $x_j\in (p_{i_j},p_{i_j+1})$ possibly $l$ or $s$ equal to zero.
    \[
        \chop_f(x) = \left(\Big(\underbrace{0,\ldots 0}_{l+s \text{ times}},\hcoord_m(x_{l+s+1}),\ldots,\hcoord_m(x_n)\Big), f\circ \iota_{((i_1+1),\ldots ,(i_{l}+1),(i_{l+1}),\ldots (i_n) )}\right)
    \]

    Since all elements of $(y^{(j)})$ are contained in $W$ it follows that
    $\mathbf{\hidx}^n_m(y^{(j)})=(i_1,\ldots,i_n)$ for all $j$. We calculate
    \begin{align*}
        \lim{j\to \infty} \chop_f(y^{(j)}) & = \lim{j\to \infty} \left( \left(\hcoord_m(y_1^{(j)}),\ldots,\hcoord_m(y_n^{(j)})\right), f\circ\iota_{(i_1,\ldots,i_n)} \right)                                                                             \\
                                           & = \left( \Big(\underbrace{1, \dots, 1}_{l \text{ times}},\underbrace{0, \dots, 0}_{s \text{ times}},\hcoord_m(x_{l+s+1}),\ldots,\hcoord_m(x_n)\Big), f\circ\iota_{(i_1,\ldots i_n)} \right)                  \\
                                           & = \left( \Big(\underbrace{0, \dots, 0}_{s \text{ times}},\hcoord_m(x_{l+s+1}),\ldots,\hcoord_m(x_n)\Big), \csd_1^1\circ \cdots \circ\csd_l^1(f\circ\iota_{(i_1,\ldots i_n)}) \right)                         \\
                                           & = \left( \Big(\underbrace{0, \dots, 0}_{s \text{ times}},\hcoord_m(x_{l+s+1}),\ldots,\hcoord_m(x_n)\Big), \csd_1^0\circ \cdots \circ\csd_l^0(f\circ\iota_{(i_1+1,\ldots i_l+1,i_{l+1},\ldots, i_n)}) \right) \\
                                           & = \left(\Big(\underbrace{0,\ldots 0}_{l+s \text{ times}},\hcoord_m(x_{l+s+1}),\ldots,\hcoord_m(x_n)\Big), f\circ \iota_{((i_1+1),\ldots ,(i_{l}+1),(i_{l+1}),\ldots (i_n) )}\right)                          \\
                                           & = \chop_f(x).
    \end{align*}
    This shows that $\chop_f$ is continuous. Recall that $|N_m X|=\left( \bigsqcup_{n\in\NN} [0,1]^n\times N_m X(\one^n) \right)/\sim$; to show that $\chop_m$ is well-defined we have to show that if $(x,f)\sim(y,g)$, then $\chop_f(x)\sim\chop_g(y)$, which is equivalent to showing that
    \begin{align}
        \left(\,\chop_{ \csd_i^\ve f}(t_1,\ldots ,t_{n-1})\right) & \sim \left(\,\chop_f(t_1,\ldots, t_{i-1}, \ve, t_i,\ldots, t_{n-1})\right)        \label{eq:chopfac}                                \\
        \left(\,\chop_{ \css_i f}(t_1,\ldots ,t_{n+1})\right)     & \sim \left(\,\chop_f(t_1,\ldots, t_{i-1}, t_{i+1},\ldots, t_{n+1})\right)             \label{eq:chopdeg}                            \\
        \left(\,\chop_{ \csg_i^\ve f}(t_1,\ldots ,t_{n+1})\right) & \sim \left(\,\chop_f(t_1,\ldots, t_{i-1}, m^\ve(t_i,t_{i+1}), t_{i+2},\ldots, t_{n+1})\right)                    \label{eq:chopcon}
    \end{align}
    for all $1\leq i\leq n$, $\ve\in\{0,1\}$, $t_i\in[0,1]$ and $f\in N_m X(\one^n)$. We calculate \eqref{eq:chopfac}, the equalities \eqref{eq:chopdeg} and \eqref{eq:chopcon} can be shown with analogous calculations.
    \begin{align*}
        \left(\,\chop_f(t_1,\ldots, t_{i-1}, \ve, t_i,\ldots, t_{n-1})\right) & =\left( \mathbf{\hcoord}^n_m(t_1,\ldots, t_{i-1},\ve, t_i,\ldots, t_{n-1}), f\circ \iota_{\mathbf{\hidx}_m^n(t_1,\ldots, t_{i-1},\ve, t_i,\ldots, t_{n-1})}\right)
    \end{align*}
    Since $\hidx_m(\ve)=\ve\cdot (m-1)$ and $\hcoord_m(\ve)=\ve$, this is equal to
    \begin{align*}
         & =\left( \mathbf{\hcoord}^{i-1}_m(t_1,\ldots, t_{i-1}),\ve,\mathbf{\hcoord}_m^{n-i}(t_i,\ldots, t_{n-1}), f\circ \iota_{\mathbf{\hidx}_m^{i-1}(t_1,\ldots, t_{i-1}),\ve\cdot(m-1),\mathbf{\hidx}_m^{n-i}(t_i,\ldots, t_{n-1})}\right)           \\
         & = \left( \mathbf{\hcoord}^{i-1}_m(t_1,\ldots, t_{i-1}),\mathbf{\hcoord}_m^{n-i}(t_i,\ldots, t_{n-1}),\csd_i^\ve (f\circ \iota_{\mathbf{\hidx}_m^{i-1}(t_1,\ldots, t_{i-1}),\ve\cdot(m-1),\mathbf{\hidx}_m^{n-i}(t_i,\ldots, t_{n-1})) }\right) \\
         & =\left( \mathbf{\hcoord}^{n-1}_m(t_1,\ldots, t_{n-1}),\csd_i^\ve(f)\circ \iota_{\mathbf{\hidx}_m^{n-1}(t_1,\ldots, t_{n-1}) }\right)                                                                                                           \\
         & =\chop_{\csd_i^\ve(f)}(t_1,\ldots,t_{n-1}).
    \end{align*}

\end{proof}
Recall that the map $\gref_m: I_{m+2}\to I_m$ induces a map of cubical sets $\cref_m: N_m X \to N_{m+2} X$ given in degree $n$ by precomposition with $\gref_m^n: I_{m+2}^n \to I_m^n$.
\begin{proposition}
    The following diagram commutes
    \[
        \begin{tikzcd}
            \vert N_m X\vert  \ar[r,"|\cref_m|"] \ar[dr,"\chop_m"] & \vert N_{m+2} X\vert \ar[d,"\chop_{m+2}"] \\
            &  \vert N_1 X\vert.
        \end{tikzcd}
    \]
\end{proposition}
\begin{proof}
    First observe that it holds that $\gref_m(\hidx_{m+2}(t)) - \hidx_m(t)$ is non-zero only for $t\in [1-\frac{1}{2^{k+2}},1]$ and
    $\hcoord_{m+2}(t)- \hcoord_m(t)$ vanishes everywhere except on the intervals $(\frac{1}{2^{k+3}},\frac{1}{2^{k+2}})$ and $[1-\frac{1}{2^{k+2}},1-\frac{1}{2^{k+3}})$. On the first of these $\hcoord_m$ is $0$ while $\hcoord_{m+2}$ is not, and on the second $\hcoord_m$ is $1$ while $\hcoord_{m+2}$ is not. Note that therefore, if $(t_1,\ldots t_n)\in [0,1]^n$ contains no coordinate in one of the intervals $(\frac{1}{2^{k+3}},\frac{1}{2^{k+2}})$ and $[1-\frac{1}{2^{k+2}},1]$, then $\chop_{m+2}(|\cref_m|((t_1,\ldots,t_n),f))=\chop_m(((t_1,\ldots,t_n),f))$. Now assume we are given $((t_1,\ldots t_n),f)\in | N_m X|$ with $t_i$ in one of the latter two intervals. Calculating
    \begin{align*}
        \chop_{m+2}(|\cref_m|((t_1,\ldots,t_n),f)) & = \chop_{m+2}((t_1,\ldots,t_n),f \circ \gref_m^n)                                                                        \\
                                                   & = (\mathbf{\hcoord}_{m+2}^n(t_1, \ldots, t_n), f \circ \gref_m^n\circ \iota_{\mathbf{\hidx}_{m+2}^n(t_1, \ldots, t_n)} )
    \end{align*}
    Assume first $t_i\in (\frac{1}{2^{k+3}},\frac{1}{2^{k+2}})$, then $\hidx_{m+2}(t_i)=0$ and hence $\gref_m^n \circ \iota_{\mathbf{\hidx}_{m+2}^n(t_1, \ldots, t_n)} $ is degenerate on the $i$-th coordinate, since $\gref_m(0)=\gref_m(1)=0$.
    Similarly, if $t_i\in [1-\frac{1}{2^{k+2}},1]$, then $\hidx_{m+2}(t_i)=m+1$ and hence $\gref_m^n \circ \iota_{\mathbf{\hidx}_{m+2}^n(t_1, \ldots, t_n)} $ is degenerate on the $i$-th coordinate, since $\gref_m(m+1)=\gref_m(m+2)=m$.
    Thus, in either case, it holds that
    \begin{align*}
         & \chop_{m+2}(|\cref_m|((t_1,\ldots,t_n),f)) \\ &= \left(\mathbf{\hcoord}_{m+2}^{n-1}(t_1,\ldots,t_{i-1},t_{i+1},\ldots,t_n),\csd_i^\ve\left(f \circ \gref^{n}_{m}\circ \iota_{\mathbf{\hidx}_{m+2}^{n-1}(t_1,\ldots,t_n)} \right)\right)
    \end{align*}
    and
    \begin{align*}
        \chop_m(((t_1,\ldots,t_n),f)) & = \left(\mathbf{\hcoord}_{m}^{n-1}(t_1,\ldots,t_{i-1},t_{i+1},\ldots,t_n),\csd_i^\ve\left(f \circ \iota_{\mathbf{\hidx}_{m}^{n-1}(t_1,\ldots,t_{i-1},t_{i+1},\ldots,t_n)} \right)\right)
    \end{align*}
    where $\ve=0$ if $t_i\in (\frac{1}{2^{k+3}},\frac{1}{2^{k+2}})$ and $\ve=1$ if $t_i\in [1-\frac{1}{2^{k+2}},1]$.
    We iteratively apply this argument until all coordinates do not lie in one of the intervals $(\frac{1}{2^{k+3}},\frac{1}{2^{k+2}})$ and $[1-\frac{1}{2^{k+2}},1]$.
\end{proof}
The topological realization functor $|\cdot|:\cSet\to \Top$ admits a right adjoint, namely the cubical singular functor $\mathrm{Sing}_\square$ assigning to a topological space $X$ the cubical set given in degree $n$ by the set of continuous maps $[0,1]^n\to X$. Hence $\vert \cdot \vert $ preserves colimits and therefore
\[
    |NX|=\colim{}\left(|N_1 X|\xrightarrow{|\crefl_1|} |N_2 X|\xrightarrow{|\crefr_2|}\cdots  \right).
\]
For even $m$, it holds by definition that $\chop_{m+1}\circ|\crefr_m| = \chop_m$.
For odd $m$, by $\crefr_{m+1}\circ\crefl_m = \cref_m$, we have
\[
    \chop_{m+2}\circ|\crefr_{m+1}\circ\crefl_m| = \chop_{m+2}\circ|\cref_m| = \chop_m.
\]
Therefore, by the universal property of the colimit, there is a map $\chop: |NX|\to |N_1 X|$ such that the following diagram commutes
\[
    \begin{tikzcd}
        \vert N_1 X\vert \arrow[r, "|\crefl_1|"]\arrow[dr, "\chop_1"]  &  \vert N_2 X\vert\arrow[r, "|\crefr_2|"]\ar[d,"\chop_2"] &\cdots \ar[r] & \vert NX\vert \ar[dashed, "\chop",dll]  \\
        & \vert N_1 X\vert
    \end{tikzcd}
\]
For $m\geq 1$, denote with $\cref^m:N_1 X\to N_m X$ the composition $\cref_{m-2}\circ \cref_{m-4}\circ \cdots \circ \cref_1$ for $m$ odd and $\crefl_{m-1}\circ\cref^{m-1}$ for $m$ even, with the convention $\cref^2= \crefl_1$.
\begin{proposition}\label{prop:homotopy_inverse_Psi}
    The map $\chop_m: |N_m X|\to |N_1 X|$ is a homotopy inverse of $|\cref^m|$ and $\chop$ is a homotopy inverse of $|\cref|\from |N_1 X|\to |NX|$.
\end{proposition}
\begin{proof}
    Firstly observe that $\chop_1:|N_1 X|\to |N_1 X|$ is given for $((t_1,\ldots t_n),f)\in |N_1 X|$ as
    \[
        \chop_1(((t_1,\ldots t_n),f))=(\mathbf{\hcoord}^n_1(t_1,\ldots,t_n),f)
    \]
    and recall the map $\hcoord_1$ is given by
    \[
        \hcoord_1(t)=\begin{cases}
            0                                & t<\frac{1}{4}                   \\
            2\cdot\left(t-\frac{1}{4}\right) & \frac{1}{4}\leq t < \frac{3}{4} \\
            1                                & \frac{3}{4}\leq t,
        \end{cases}
    \]
    which is homotopic to the identity on $[0,1]$ relative $\{0,1\}$, e.g. via $H(s,t)= s\cdot \hcoord_1(t) + (1-s)\cdot t$. Applying $H$ to each coordinate $t_i$ gives rise to a homotopy of $\chop_1$ to $\id{|N_1 X|}$, so $\chop_1 \simeq \id{|N_1 X|}$. Since $\chop_m\circ|\cref^m|=\chop_1$ and similarly $\chop\circ|\cref|=\chop_1$, it follows that $\chop_m$ and $\chop$ are homotopy left inverses to $|\cref^m|$ and $|\cref|$ respectively. Since both $|\cref^m|$ and $|\cref|$ are weak equivalences, it follows from the fact that weak equivalences satisfy the two out of three property that $\chop_m$ and $\chop$ are weak equivalences, and thus it follows from Whitehead's theorem that also $|\cref^m|\circ \chop_m$ and $|\cref|\circ \chop$ are homotopic to the identity.
\end{proof}

\subsection{On relative A-homotopy groups}
We start by observing that, for a graph map $f:(I_m^n,\del I_m^n,|\horn^n|_{m})\to(X,A,*)$, the restriction of $\chop_f\from[0,1]^n \to |N_1 X|$ to the boundary $\del[0,1]^n$ is contained in $|N_1 A|$ and that $\chop_f$ maps the Horn $|\horn^n|$ to the basepoint $*$. The following theorem states that this identification gives rise to an isomorphism between relative $A$-homotopy groups of $(X,A,*)$ and the relative classical homotopy groups of $(|N_1X|,|N_1A|,*)$.
\begin{theorem}\label{lem:relativemagic}
    Let $X$ be a graph and $A$ be a subgraph of $X$ and $f:(I_m^n,\del I_m^n,|\horn^n|_{m})\to(X,A,*)$ be a representative of a homotopy class $[f]\in\AA_n(X,A,*)$ and let
    \begin{align*}
        \chop\circ|\bar{f}|: \left( [0,1]^n,\del[0,1]^n,|\horn^n|  \right) & \la (|N_1 X|,|N_1 A|,*)           \\
        (t_1,\ldots, t_n)                                                  & \lm \chop(((t_1,\ldots, t_n),f)).
    \end{align*}
    Then $\hat{\magic}_n:\AA_n(X,A,*)\cong \pi_n(|N_1 X|,|N_1 A|,*)$ with $[f]\mapsto [\chop\circ|\bar{f}|]$ is an isomorphism, in the sense that there is an isomorphism of long exact sequences
    \[
        \begin{tikzcd}[column sep=1.4em, row sep=2.5em]
            \cdots \arrow[r]
            & \AA_n(A,*) \arrow[r, "\iota_*"] \arrow[d, "\chop_*\circ\magic_n"]
            & \AA_n(X,*) \arrow[r, "j_*"] \arrow[d, "\chop_*\circ\magic_n"]
            & \AA_n(X,A,*) \arrow[r, "\delta"] \arrow[d, "\hat{\magic}_n"']
            & \AA_{n-1}(A,*) \arrow[r] \arrow[d, "\chop_*\circ\magic_{n-1}"]
            & \cdots
            \\
            \cdots \arrow[r]
            & \pi_n(|N_1 A|,*) \arrow[r, "\iota_*"]
            & \pi_n(|N_1 X|,*) \arrow[r, "j_*"]
            & \pi_n\big(|N_1 X|,|N_1 A|,*\big) \arrow[r, "\delta"]
            & \pi_{n-1}(|N_1 A|,*) \arrow[r]
            & \cdots.
        \end{tikzcd}
    \]
\end{theorem}
\begin{proof}
    Recall that the isomorphism $\magic_n:\AA_n(X,*)\to\pi_n(|NX|, *)$ is given by assigning to a homotopy class $[f]\in \AA_n(X,*)$ with representative $f:(I_m^n,\del I_m^n)\to (X,*)$ the homotopy class of
    \begin{align*}
        |\bar{f}|:([0,1]^n,\del[0,1]^n) & \la (|NX|,*)           \\
        (t_1,...,t_n)                   & \lm ((t_1,...,t_n),f).
    \end{align*}

    Now, since
    \[
        |\bar{f}|(t_1,\ldots, t_{i-1},\ve,t_{i+1},\ldots, t_n)= ((t_1,\ldots, t_{i-1},t_{i+1},\ldots, t_n),\csd_i^\ve f),
    \]
    it follows that if $f:(I_m^n,\del I_m^n,|\horn^n|_{m})\to(X,A,*)$ is a map of triples, then also $|\bar{f}|:\left( [0,1]^n,\del[0,1]^n,\hg{}{n}  \right)\to (|NX|,|NA|,*) $ is a  map of triples and similarly a homotopy
    \[
        H:(I_m^n\bprod I_k,\del I_m^n\bprod I_k,\hg{m}{n}\bprod I_k) \to (X,A,*)
    \]
    corresponds to a topological homotopy
    \[
        |\bar{H}|:([0,1]^{n+1},\del [0,1]^n\times [0,1],|\horn^n|\times [0,1]) \to (|NX|,|NA|,*).
    \]
    Therefore, the assignment $[f]\mapsto [|\bar{f}|]$ gives rise to a map of sets $\magic_n:\AA_n(X,A,*)\to \pi_n(|NX|,|NA|,*)$.

    We will show that this map is a group homomorphism for $n\geq 2$ (and a map of pointed sets for $n=1$). For that, let us consider $f,g:(I_m^n,\del I_m^n,|\horn^n|_{m})\to(X,A,*)$. Consider the maps $F,G:I_{2m}^{n+1}\to X$ defined by
    \[
        F(v_1,\ldots, v_{n+1})\letbe \begin{cases}
            f(v_1,\min(v_2,m),\ldots ,\min(v_n,m))                 & v_{1}\leq m       \\
            g(\min(v_1-m,v_{n+1}),\min(v_2,m),\ldots, \min(v_n,m)) & \text{otherwise}.
        \end{cases}
    \]
    and
    \[
        G(v_1,\ldots,v_{n+1})\letbe
        g(\min(\max(v_{n+1},v_1),m),\min(v_2,m),\ldots, \min(v_n,m)).
    \]
    Now consider the map $H:([0,1]^{n+1},\del [0,1]^{n}\times [0,1],|\hg{}{n}|\times [0,1])\to (|NX|,|NA|,*)$
    \[
        H(t_1,\ldots, t_{n+1})=\begin{cases}
            |\bar{F}|(2t_1,t_2,\ldots, t_{n+1})   & t_1\leq \frac{1}{2}  \\
            |\bar{G}|(2t_1-1,t_2,\ldots, t_{n+1}) & t_1\geq \frac{1}{2}.
        \end{cases}
    \]
    Observe that $H(t_1,\ldots, t_n,1)=|\overline{f\cdot g}|(\min(2t_1,1),t_2,\ldots, t_n)$, whereas $(t_1,\ldots, t_n)\mapsto H(t_1,\ldots, t_n,0)$ is homotopic to $|\bar{f}|\cdot |\bar{g}|$. Thus we overall get that $\magic_n([f\cdot g])=\magic_n([f])\cdot \magic_n([g])$.
    Now consider the diagram
    \[
        \begin{tikzcd}[column sep=1.4em, row sep=2.5em]
            \AA_n(X,*) \arrow[r, "j_*"] \arrow[d, "\magic_n"]
            & \AA_n(X,A,*) \arrow[r, "\delta"] \arrow[d, "\magic_n"']
            & \AA_{n-1}(A,*) \arrow[d, "\magic_{n-1}"]
            \\
            \pi_n(|NX|,*) \arrow[r, "j_*"]
            & \pi_n(|NX|,|NA|,*) \arrow[r, "\delta"]
            & \pi_{n-1}(|NA|,*)
        \end{tikzcd}
    \]
    It is immediate from the definition of $\magic_n$ and the maps $j_*$ that the left square commutes. Similarly, in either case the map $\delta$ is given by restriction to the top face, and hence the right square also commutes. Furthermore the map of pairs $\chop: (|NX|,|NA|,*) \to (|N_1 X|,|N_1 A|,*)$ gives rise to a map of long exact sequences of homotopy groups. Putting it all together we get that the following diagram commutes

    \[
        \begin{tikzcd}[column sep=1.4em, row sep=2.5em]
            \cdots \arrow[r]
            & \AA_n(A,*) \arrow[r, "\iota_*"] \arrow[d, "\magic_n"]
            & \AA_n(X,*) \arrow[r, "j_*"] \arrow[d, "\magic_n"]
            & \AA_n(X,A,*) \arrow[r, "\delta"] \arrow[d, "\magic_n"']
            & \AA_{n-1}(A,*) \arrow[r] \arrow[d, "\magic_{n-1}"]
            & \cdots
            \\
            \cdots \arrow[r]
            & \pi_n(|NA|,*) \arrow[r, "\iota_*"] \arrow[d, "\chop_*"]
            & \pi_n(|NX|,*) \arrow[r, "j_*"] \arrow[d, " \chop_*"]
            & \pi_n\big(|NX|,|NA|,*\big) \arrow[r, "\delta"] \arrow[d, " \chop_*"]
            & \pi_{n-1}(|NA|,*) \arrow[r] \arrow[d, "\chop_*"]
            & \cdots
            \\
            \cdots \arrow[r]
            & \pi_n(|N_1 A|,*) \arrow[r, "\iota_*"]
            & \pi_n(|N_1 X|,*) \arrow[r, "j_*"]
            & \pi_n\big(|N_1 X|,|N_1 A|,*\big) \arrow[r, "\delta"]
            & \pi_{n-1}(|N_1 A|,*) \arrow[r]
            & \cdots.
        \end{tikzcd}
    \]
    For $n\geq 2$, applying the five lemma twice, once for $\magic$ and once for $\chop$ we get that the vertical composition $\hat{\magic}_n\letbe\chop_* \circ \magic_n:\AA_n(X,A,*) \to \pi_n\big(|N_1 X|,|N_1 A|,*\big)$ is an isomorphism as desired.

    In degree $n=1$, we cannot apply the five lemma.

    We will first show surjectivity, for that let $f:([0,1],\{0,1\},\{0\})\to (|N_1 X|,|N_1 A|,*)$ be a representative of a homotopy class in $\pi_1(|N_1 X|,|N_1 A|,*)$. By cellular approximation, we can choose $f$ to be a cellular map. However, $1$-skeleton of $|N_1 X|$ is contains two one cell for each edge of $X$, one for each orientation. Listing the vertices in order of appearance after walking through $f$ can hence be associated to a graph map $g:I_m\to X$. Then $f$ can easily checked to by homotopic to $\chop\circ |\bar{g}|$.

    What is left to show is injectivity, assume we are given two paths
    $[f],[g]\in\AA_1(X,A,*)$ that get mapped by $\hat{\magic}_1$ to the same homotopy class. Then their endpoint has to be in the same connected component $A$, denote it by $B$. We may assume the representatives end in the same endpoint $b\in B$, by cellular approximation, we can even assume that $f,g$ satisfy that the homotopy from $\chop|\bar{f}|$ to $\chop|\bar{g}|$ is a homotopy relative endpoints. We can therefore consider the concatenation $f\cdot g^{-1}$, which first traverses $f$ and then $g$ in the oposite direction. This is a loop $\pi_1(X)$, however, since $\hat{\magic}_1([f])=\hat{\magic}_1([g])$, this loop $f\cdot g^{-1}$ is contractible relative endpoints. Therefore also the concatenation $g^{-1}\cdot f$, which is a loop starting and ending in $b$, is contractible relative endpoints. But now we have that
    \[
        f \simeq f \cdot c_b \simeq f \cdot g^{-1} \cdot g \simeq c_{*}\cdot g
    \]
    Where $\simeq$ denotes homotopy relative endpoints and $c_b$ and $c_*$ denote the constant paths.
\end{proof}
\subsection{The relative Hurewicz homomorphism}
The discrete Hurewicz homomorphism $\Hurewicz_n:\AA_n(X,*_X)\to H_n(X)$ of a graph $X$ has been introduced in \cite{Lutz:Higher} and has been shown to fulfill a discrete version of the classical Hurewicz theorem in \cite{Chris:Revisited}.
For a representative $f:I_m^n\to X$ of a homotopy class $[f] \in \AA_n(X,*_X)$ the Hurewicz homomorphism is given by
\[
    \Hurewicz_n([f])\letbe \left[ \sum_{\mathbf{v}\in [m-1]^n} f\circ \iota_\mathbf{v} \right] \in H_n(X),
\]
where, as before, the map $\iota_\mathbf{v}:I_1^n\to I_m^n$ is given by $\mathbf{x}\mapsto \mathbf{v}+\mathbf{x}$. Observe that the discrete Hurewicz homomorphism can also be obtained as the composition
\[
    \AA_n(X,*_X)\xrightarrow{\magic_n}\pi_n(|NX|, *_X)\xrightarrow{\chop_*}\pi_n(|N_1 X|, *_X)\xrightarrow{\classHurewicz}H_n(|N_1 X|)\cong H_n(X),
\]
where $\classHurewicz$ is the classical Hurewicz homomorphism.

Similarly, we can consider a discrete relative Hurewicz homomorphism $\relHurewicz_n:\AA_n(X,A,*)\to H_n(X,A)$ as the composition
\[
    \AA_n(X,A,*)\xrightarrow{\hat{\magic}_n}\pi_n(|N_1 X|,|N_1 A|,*)\xrightarrow{\classHurewicz}H_n(|N_1 X|,|N_1 A|)\cong H_n(X,A).
\]
Hence, as an immediate consequence of the classical theory of a pair we get
\begin{theorem}
    Let $n\geq 2$ and $(X,A)$ be a pair of simply connected graphs such that $H_i(X,A)=0$ for all $i<n$. Then $\AA_i(X,A,*)=0$ for all $i<n$ and $\relHurewicz_n:\AA_n(X,A,*)\to H_n(X,A)$ is an isomorphism.
\end{theorem}
\begin{proof}
    Follows from \cite{Dieck:AT}, Theorem 20.1.3.
\end{proof}
\section{Constructing graphs for given A-homotopy groups}
\subsection{On coverings}
We will start by giving an equivalent definition of a covering. We will use this to define a graph analogue of the classical notion of a free and properly discontinuous action, in the sense that the quotient map of such an action is a covering. The author felt uncomfortable with the word ``discontinuous'' in the graph setting, for graphs such an action will be called a \fa{freely sparse action}. We start by giving an equivalent definition of a covering.
\begin{lemma}\label{lem:covering_Umg}
    A graph map $p:\tilde{X}\to X$ is a covering if and only if for every $x\in X$ it holds that
    \[
        p^{-1}(\Umg{2}{x})\cong \bigsqcup_{y\in p^{-1}(x)} \Umg{2}{y},
    \]
    where the isomorphism is given by the restriction of $p$ to each component indexed by $y\in p^{-1}(x)$.
\end{lemma}
\begin{proof}
    Firstly, assume that $p:\tilde{X}\to X$ satisfies the condition above. We will show that $p$ is a covering map. The fact that for every $x\in X$ and $y\in p^{-1}(x)$ it holds that $p\vert_{\Umg{1}{y}}\colon \Umg{1}{y}\to\Umg{1}{x}$ is an isomorphism is immediate, since $\Umg{1}{y}\subset \Umg{2}{y}$ and the restriction of $p$ to $\Umg{2}{y}$ is an isomorphism. It remains to be shown that $p$ has the right lifting property with respect to $\eta:I_3\to I_1^2$, assume we are given a diagram
    \[
        \begin{tikzcd}
            I_3 \arrow[d, "\eta"]\arrow[r,"\tilde{f}"]
            & \tilde{X} \arrow[d, "p"]
            \\
            I_1^2 \arrow[r, "f"]
            & X.
        \end{tikzcd}
    \]
    Let $x\letbe f(0,0)$ and $y\letbe \tilde{f}(0)$. Recall that $\eta(0)=(0,0)$, therefore it holds that $x=p(y)$. Also, observe that $f(I_1^2)\subset \Umg{2}{x}$, and, by the commutativity of the diagram, $\tilde{f}(I_3)\subseteq \Umg{2}{y}$. Consider the composition $ h\letbe(p\vert_{\Umg{2}{y}})^{-1}\circ f\from I_1^2\to \tilde{X} $, then the diagram
    \[
        \begin{tikzcd}
            I_3 \arrow[d, "\eta"]\arrow[r,"\tilde{f}"]
            &\Umg{2}{y} \subseteq \tilde{X} \arrow[d, "p"]
            \\
            I_1^2 \arrow[r, "f"]\arrow[dashed,ru,"h"]
            & \Umg{2}{x} \subseteq X.
        \end{tikzcd}
    \]
    commutes and hence $h$ is a lift as desired.

    Conversely, assume that $p$ is a covering map, let $x\in X$ and $y\in p^{-1}(x)$, we will show that $p\vert_{\Umg{2}{y}}:\Umg{2}{y}\to\Umg{2}{x}$ is an isomorphism. We do this by constructing an inverse, let $z\in \Umg{2}{x}$, then there is a map $\gamma_{z}:I_2\to X$ with $\gamma_z(0)=x$ and $\gamma_z(2)=z$ since $z \in \Umg{2}{x}$. Since $p$ is a covering map there exists a unique lift $\tilde{\gamma}_{z}:I_2 \to \tilde{X}$ of $\gamma_{z}$ starting at $y$ by \cite[4.5]{Udit:Fundamentalgroupoid}. Note that $\tilde{\gamma}_{z}(2)\in\Umg{2}{y}$, we will show that $\tilde{\gamma}_{z}(2)$ is invariant under the choice of path $\gamma$, for this, let $\gamma'$ be another path starting at $x$ and ending at $z$. Then consider the map $f:I_1^2\to X$ that sends $(0,0)\mapsto x$, $(1,0)\mapsto \gamma_z(1)$, $(0,1)\mapsto \gamma'_z(1)$ and $(1,1)\mapsto z$. This map admits a lift to $\tilde{X}$, using the uniqueness of the lifts of paths, it then has to hold that $\tilde{\gamma}'_z(2)=\tilde{\gamma}_z(2)$.
    Thus, the following map is well defined
    \begin{align*}
        q\from\Umg{2}{x} & \la \Umg{2}{y}             \\
        z                & \lm \tilde{\gamma}_{z}(2).
    \end{align*}

    We show that $q$ is a graph map, for that let $e=\{z_1,z_2\}\in E(\Umg{2}{x})$. Then either $z_1$ or $z_2$ is in $\Umg{1}{x}$, assume without loss of generality $z_1\in \Umg{1}{x}$. Consider the maps
    \begin{align*}
        f\from I_1^2 & \la X               & \tilde{f}\from I_3 & \la \tilde{X}               \\
        (0,0)        & \lm z_1             & 0                  & \lm q(z_1)                  \\
        (1,0)        & \lm x               & 1                  & \lm y                       \\
        (1,1)        & \lm \gamma_{z_2}(1) & 2                  & \lm \tilde{\gamma}_{z_2}(1) \\
        (0,1)        & \lm z_2             & 3                  & \lm \tilde{\gamma}_{z_2}(2)
    \end{align*}
    The existence of a lift of
    \[
        \begin{tikzcd}
            I_3 \arrow[d, "\eta"]\arrow[r,"\tilde{f}"]
            & \tilde{X} \arrow[d, "p"]
            \\
            I_1^2 \arrow[r, "f"]
            & X.
        \end{tikzcd}
    \]
    implies that there is an edge $\{q(z_1),q(z_2)\} \in E(p^{-1}(\Umg{2}{x}))$, hence $q$ is a graph map. By construction $p\circ q=\id{\Umg{2}{x}}$ and $q(x)=y$, so $q$ is inverse to $p\vert_{\Umg{2}{y}}$ and the latter is an isomorphism onto $\Umg{2}{x}$.
    We have thus shown that $p^{-1}(\Umg{2}{x})$ is the union of 2-neighborhoods around the elements of $p^{-1}(x)$, it remains to show that this union is disjoint. For that, assume that there is a vertex $v\in \Umg{2}{y_1}\cap \Umg{2}{y_2}$ for $y_1, y_2\in p^{-1}(x)$. Then there is a graph map $\hat{f}:I_3\to \tilde{X}$ with $\hat{f}(0)=y_1$, $\hat{f}(2)=v$ and $\hat{f}(3)\in \Umg{1}{y_2}$. Since $\hat{f}(3)\in \Umg{1}{y_2}$, there is an edge $\{p\circ\hat{f}(3),x\}\in \E{X}$, we can hence extend $p\circ\hat{f}$ to a graph map $f:C_4=I_1^2\to X$ such that
    \[
        \begin{tikzcd}
            I_3 \arrow[d, "\eta"]\arrow[r,"\tilde{f}"]
            & \tilde{X} \arrow[d, "p"]
            \\
            I_1^2 \arrow[r, "f"]
            & X.
        \end{tikzcd}
    \]
    commutes. Now the existence of a lift of $f$ to $\tilde{X}$ implies that the distance between $\hat{f}(3)$ and $y_1$ is $1$ and hence both $y_1$ and $y_2$ lie in the 1-neighborhood of $\hat{f}(3)$. Thus, since the restriction of $p$ to 1-neighborhoods is an isomorphism, it follows that $y_1=y_2.$
\end{proof}
\begin{definition}[freely sparse action]
    Let $G$ be a group and $X$ a graph. A left action of $G$ on $V(X)$ is called \fa{freely sparse} if the following conditions are satisfied.
    \begin{enumerate}
        \item For every edge $\{x_1,x_2\} \in E(X)$ and every $g\in G$, it holds that $\{g\triangleright x_1 ,g\triangleright x_2\}\in E(X)$
        \item For every $x\in X$ and every $g\in G$, the intersection $\Umg{2}{x}\cap \Umg{2}{g\triangleright x}$ is nonempty if and only if $g$ is the identity in $G$.
    \end{enumerate}
\end{definition}
\begin{remark}
    Note that condition (2) in particular implies that the action of $G$ on $V(X)$ is free: if $g\triangleright x = x$, then $x \in \Umg{2}{x}\cap \Umg{2}{g\triangleright x} \neq \emptyset$, hence $g=e$.
\end{remark}
\begin{lemma}
    A left group action of a group $G$ on $V(X)$, the vertex set of a graph $X$, is freely sparse if and only
    \begin{enumerate}
        \item left multiplication with $g$ is a graph map for every $g\in G$ and
        \item for every $x\in V(X)$, it holds that $d(x,g\triangleright x)\geq 5$ for all $g\in G\setminus\{e\}$.
    \end{enumerate}
\end{lemma}
\begin{proof}
    (1) is simply a restatement of the first condition in the definition of a freely sparse action. (2) is a reformulation of the second condition: The intersection of the 2-neighboorhoods of $x$ and $g\triangleright x$ is nonemtply of and only if the distance between $x$ and $g\triangleright x$ is lesser than $5$.
\end{proof}
\begin{lemma}\label{lem:FundGrpCovering}
    Let $(X,*_X)$ be a connected, based graph with a freely sparse left action of a group $G$, then the quotient map $q:(X,*_X)\to (G\backslash X,*_ {G\backslash X})$ is a regular covering and there is an isomorphism
    \[
        \Phi:\faktor{\AA_1(G\backslash X,*_{G\backslash X})}{q_*(\AA_1(X,*_X))}\cong G
    \]
    such that $\Phi$ is compatible with the action of $G$ on the fiber $q^{-1}(*_{G\backslash X})$ and the monodromy action in the sense that for every $x\in q^{-1}(*_{G\backslash X})$ and every $[\gamma]\in \AA_1(G\backslash X,*_{G\backslash X})$, it holds that
    \[
        (h \triangleright *_X) \triangleleft [\gamma] = (h\cdot \Phi([\gamma]q_*(\AA_1(X,*_X))))\triangleright ( *_X).
    \]
\end{lemma}
\begin{proof}
    Since $G$ acts freely and sparsely on $X$, for every $x\in X$, the preimage of the neighborhood can be characterized as
    \[
        q^{-1}(\Umg{2}{q(x)})= \bigsqcup_{g\in G} \Umg{2}{g\triangleright x}.
    \]
    Therefore, by \ref{lem:covering_Umg}, it holds that $q$ is a covering. Furthermore, there is a canonical identification between $q^{-1}(*_{G\backslash X})$ and $G$ where we identify $x\in q^{-1}(*_{G\backslash X})$ with the element $g\in G$ such that $g\triangleright *_X=x$. Note that this element $g$ is unique for a given $x$ since the action of $G$ on $X$ is free. Using this, we consider the map
    \begin{align*}
        \AA_1(G\backslash X,*_{G\backslash X}) & \la q^{-1}(*_{G\backslash X}) \\
        [\gamma]                               & \lm *_X\triangleleft [\gamma]
    \end{align*}
    where $*_X\triangleleft [\gamma]$ denotes the monodromy action of $[\gamma]$ on $*_X$, i.e., the endpoint of the unique lift of $\gamma$ starting at $*_X$. Using the identification of $q^{-1}(*_{G\backslash X})$ and $G$ above, we obtain a map $\Psi:\AA_1(G\backslash X,*_{G\backslash X})\to G$ defined by the equality $\Psi([\gamma])\triangleright *_X = *_X\triangleleft [\gamma]$. We will show that $\Psi$ is an epimorphism of groups with kernel $q_*(\AA_1(X,*_X))$. Firstly, let $[\gamma_1],[\gamma_2]\in \AA_1(G\backslash X,*_{G\backslash X})$ and let $g_i=\Psi([\gamma_i])$. The unique lift of $\gamma_1\cdot\gamma_2$ starting at $*_X$ is the concatenation of the lift of $\gamma_1$ starting at $*_X$ and the lift of $\gamma_2$ starting at $*_X\triangleleft [\gamma_1] = g_1\triangleright *_X$. Let $\tilde{\gamma}_2:I_n\to X$ be the lift of $\gamma_2$ starting at $*_X$ and consider the path \begin{align*}
        \tilde{\gamma}'_2:I_n & \la X                                      \\
        t                     & \lm g_1\triangleright \tilde{\gamma}_2(t).
    \end{align*}

    Then, by the uniqueness of lifts of paths for coverings (cf.\ \cite[4.5]{Udit:Fundamentalgroupoid}), this second lift $\tilde{\gamma}'_2$ is the lift of $\gamma_2$ starting at $g_1\triangleright *_X$. Its endpoint is then
    \[
        g_1\triangleright (g_2\triangleright *_X) = (g_1g_2)\triangleright *_X~.
    \]
    This shows $\Psi([\gamma_1\cdot\gamma_2]) = g_1g_2 = \Psi([\gamma_1])\Psi([\gamma_2])$, i.e. $\Psi$ is a group homomorphism.

    To see that $\Psi$ is surjective, let $g\in G$. Since $X$ is connected, there exists a path $\tilde{\gamma}$ in $X$ from $*_X$ to $g\triangleright *_X$. Then $\gamma=q\circ \tilde{\gamma}$ is a loop in $G\backslash X$ based at $*_{G\backslash X}$, and its unique lift starting at $*_X$ is exactly $\tilde{\gamma}$. Therefore, $*_X\triangleleft [\gamma] = g\triangleright *_X$, which implies $\Psi([\gamma])=g$.

    Finally, we determine the kernel of $\Psi$. A loop $[\gamma]\in \AA_1(G\backslash X,*_{G\backslash X})$ is in $\ker(\Psi)$ if and only if $\Psi([\gamma])=e$, the identity in $G$, i.e. $*_X\triangleleft [\gamma] = e\triangleright *_X = *_X$. This holds if and only if the lift of $[\gamma]$ starting at $*_X$ is a loop in $X$, i.e. an element of $\AA_1(X,*_X)$.

    Since the group $G$ acts transitively on the fiber $q^{-1}(*_{G\backslash X})$ by deck transformations, $q$ is a regular covering. The existence of the isomorphism $\Phi$ follows from the first isomorphism theorem for groups.
\end{proof}

\subsection{Construction of $BG$}
We will give a description of Eilenberg-MacLane graphs $K(G,1)$, that we will denote $BG$. We chose this notation because of the formal similarity of the construction to Milnor's construction of classifying spaces $B\mathcal{G}$ for a topological group $\mathcal{G}$. At first, one might be inclined to generalize the construction below to group objects in the category of graphs, however, the group objects turn out to be disjoint unions of complete graphs, and hence homotopy equivalent to a discrete graph. We therefore restrict our attention to discrete graphs.
\begin{definition}\label{def:join}
    For graphs $X$ and $Y$ and a natural number $k\in \mathbb{N}$ the \fa{$k$-join of $X$ and $Y$,} $X\join{k} Y$, is the colimit of the diagram
    \[
        \begin{tikzcd}
            &X\times Y\ar[dl,"\pr{2}",swap]\ar[dr,"\iota_0"] &&X\times Y\ar[dl,"\iota_k",swap]\ar[dr,"\pr{1}"]\\
            \phantom{X\times}Y\phantom{X\times}&&X\times Y\times I_k &&\phantom{X\times}X,\phantom{X\times}
        \end{tikzcd}
    \]
    where $\iota_0$ and $\iota_k$ are the maps which map $(x,y)$ to $(x,y,0)$ and $(x,y,k)$ respectively.
\end{definition}
Put differently, $X\join{k}Y$ is the graph with vertices
\begin{align*}
    V(X\join{k}Y)=\faktor{V(X\times Y\times I_k)}{\sim}
\end{align*}
where $\sim$ is the equivalence relation generated by
\[
    (x_1,y_1,0)\sim(x_2,y_1,0) \quad\text{and}\quad
    (x_1,y_1,k)\sim(x_1,y_2,k) \quad\text{for all }(x_1,y_1),(x_2,y_2)\in V(X\times Y).
\]
The edges are given by
\[
    E(X\join{k}Y)=\left\{\, \{u,v\} \,\middle\vert\, u\neq v \text{ and there are representatives of $u$ and $v$ connected in } X\times Y\times I_k\right\}.
\]
\begin{construction}
    Let $G$ be a group. We recursively construct a sequence of graphs $E_n(G)$ equipped with a left action of $G$ as follows:
    \begin{itemize}
        \item Let $E_0(G) = G$ be the discrete graph with vertex set $G$ and no edges. The group $G$ acts on $E_0(G)$ by left multiplication.
        \item For $n \geq 1$ define $E_n(G)$ as
              \[
                  E_n(G) = G \join{n+2} E_{n-1}(G) = \left( \bigsqcup_{g\in G} \{g\}\times I_{n+2} \times E_{n-1}(G) \right) \Big/ \sim
              \]
              where $\sim$ is the equivalence relation generated by:
              \begin{align*}
                  (g,n+2,x) & \sim (g,n+2,y), \\
                  (g,0, x)  & \sim (h,0,x).
              \end{align*}
              The action of $G$ on $E_n(G)$ is given by $h\triangleright (g,k,x) := (h\cdot g,k,h\triangleright x)$.
    \end{itemize}
    We then define $EG$ as the colimit
    \[
        EG := \colim{}\left( E_0(G) \hookrightarrow E_1(G) \hookrightarrow \dots \right).
    \]
    We denote the quotients $B_n(G) \letbe G\backslash E_n(G)$ and $BG \letbe G\backslash EG$.
\end{construction}

\begin{proposition}\label{prop:actionEG}
    The action of $G$ on $E_n(G)$ for all $n\geq 0$ and on $EG$ is a freely sparse group action.
\end{proposition}
\begin{proof}
    We show this inductively on $n$. For $n=0$, $E_0(G)=G$ is a discrete graph with no edges, so the action of $G$ is freely sparse.
    Now assume that for some $n \geq 0$, the action of $G$ on $E_n(G)$ is freely sparse. Consider $E_{n+1}(G) = G \join{n+3} E_n(G)$. Let $(g,k,x)\in E_{n+1}(G)$ and $h\in G\setminus \{e\}$. And assume there is a path of length at most $4$ from $(g,k,x)$ to $h\triangleright (g,k,x) = (h\cdot g, k, h\triangleright x)$. Firstly, consider the case $k\geq 2$. Then, since $n+3\geq 3$, the $2$-neighborhood of $(g,k,x)$ is contained in the induced subgraph with vertices
    \[
        \{g\}\join{n+3} E_n(G)
    \]
    while the $2$-neighborhood of $(h\cdot g,k,x)$ is contained in the induced subgraph with vertices
    \[
        \{h\cdot g\}\join{n+3} E_n(G)
    \]
    Therefore, $k=2$ and a path from $(g,k,x)$ to $(h\cdot g,k,h\triangleright x)$ must be of the form
    \[
        (g,2,x) - (g,1,x_1) - (g,0,x_2) - (h\cdot g,1,x_2) - (h\cdot g,2,h\triangleright x),
    \]
    This implies that $x_2 \in \Umg{2}{x}$ and $x_2\in \Umg{2}{h\triangleright x}$. Therefore by induction, such a path cannot exist.

    Secondly, consider the case $k\leq 1$. Then, since $n+3 \geq 3$ every path of length at most two starting at $(g,k,x)$ and containing a cone point must terminate at this cone point. Furthermore, the only cone point that could be contained in the $2$-neighborhood of $(g,k,x)$ (of $(h\cdot g,k,h\triangleright x)$) is $(g,n+3,x)$ (resp. $(h\cdot g,n+3,h\triangleright x)$).  Thus the $2$-neighborhood of $(g,k,x)$ is contained in the subgraphs of $E_{n+1}G$ induced by
    \[
        G\join{n+3} \Umg{2}{x} \subset G\join{n+3} E_n(G).
    \]
    and the $2$-neighborhood of $h\triangleright (g,k,x)$ is contained in
    \[
        G\join{n+3} \Umg{2}{h\triangleright x} \subset G\join{n+3} E_n(G).
    \]
    Again by induction, it follows that these two induced subgraphs only intersect in the cone points, however, $(g,n+3,x)$ and $(h\cdot g,n+3,h\triangleright x)$ are distinct if $g\neq h\cdot g$, that is, if $h\neq e$. Hence, this shows that the action of $G$ on $E_{n+1}(G)$ is freely sparse.

    For the action of $G$ on $EG$, assume we are given a vertex $y\in EG$. Since $EG$ is the colimit of the sequence $E_0(G)\hookrightarrow E_1(G)\hookrightarrow \dots$, it follows that $y$ is contained in some $E_n(G)$. Assume that there is an $h\in G$ with $h\triangleright \Umg{2}{y}\cap\Umg{2}{y}\neq\emptyset$. Then there exists $k\in\mathbb{N}$ and $z\in E_k(G)$ with $h\triangleright z\in\Umg{2}{y}$ and $z\in\Umg{2}{y}$. Let $m=\max(n,k)$; since the action is freely sparse on $E_m(G)$ it follows that $h=e$.
\end{proof}
\begin{proposition}\label{prop:En_connected}
    $E_n(G)$ is $(n-1)$-connected for all $n> 0$.
\end{proposition}
\begin{proof}
    First, we will show by induction that for every $n\geq 2$, $E_n(G)$ is simply connected. For this consider a loop $\gamma:I_m\to E_n(G)$. Since $m$ is finite, it follows that $\gamma$ is contained in $X\letbe \{g_1,\ldots, g_k\}\join{n+2}E_{n-1}(G)$ for finitely many $g_1,\ldots, g_k \in G$. We will show that $X$ is simply connected: Let $X_i=\{g_1,\ldots,g_i\}\join{n+2}E_{n-1}(G)$ for $i=1,\ldots, k$. $X_1$ is a cone over $E_{n-1}(G)$, hence contractible and in particular simply connected. Now assume $X_i$ is simply connected and set $A=X_{i+1}\setminus\{(g_{i+1},n+2)\}$ and $B=X_{i+1}\setminus \{(g_l,n+2)\mid 1\leq l \leq i\}$. Then $A$ is homotopy equivalent to $X_i$ and $B \simeq \{g_{i+1}\} \join{n+2} E_{n-1}(G)$ is contractible. Their intersection $A \cap B \simeq E_{n-1}(G)$ is connected since $n \geq 2$. It follows from the Seifert-van Kampen theorem for graphs~(\ref{thm:seifertvankampen}) that $\AA_1(X_{i+1})=\AA_1(X_i)=0$. Note that this argument does not work in the case $n=1$, since the intersection $A\cap B$ would not be connected.

    We proceed by showing inductively that $E_n(G)$ is $(n-1)$-connected. Clearly the graph $E_1(G) = G\join{3}G$ is connected, i.e. $0$-connected, and $E_2(G)$ is simply connected by the first part of the proof, i.e. $1$-connected. We will now show, for $n\geq 2$, that $E_{n+1}(G)$ is $n$-connected under the assumption that $E_n(G)$ is $(n-1)$-connected. Recall that $E_{n+1}(G)=G\join{n+3}E_n(G)  =(G\times I_{n+3}\times E_n(G))/\sim$. Consider the two induced subgraphs
    \[
        A\letbe E_{n+1}(G)\setminus\left\{\,(g,n+3)\,\middle|\, g\in G  \,\right\}
    \]
    and
    \[
        B\letbe E_{n+1}(G)\setminus\left\{\,(g,0,x)\,\middle|\, x\in E_n(G), g\in G  \,\right\}.
    \]
    The distance between $A\setminus B$ and $B\setminus A$ is $n+3$. We thus obtain a partial Mayer-Vietoris sequence up to degree $n+1$. While $A$ is homotopy equivalent to $E_n(G)$, $B$ is homotopy equivalent to a disjoint union of $|G|$ vertices. The intersection $A\cap B$ has the homotopy type of $|G|$ copies of $E_n(G)$. In degrees $1\leq k\leq n-1$ it holds that $H_k(A\cap B)$, $H_k(A)$ and $H_k(B)$ vanish, therefore we have $H_k(E_{n+1}(G))=0$ for $1\leq k\leq n-1$. For degree $k=n$ consider the relevant part of the Mayer-Vietoris sequence:
    \[
        \begin{tikzcd}
            H_n(A\cap B) \ar[r,"\iota"] & H_n(A) \oplus H_n(B)  \ar[r,"j"] & H_n(E_{n+1}(G)) \ar[r] & H_{n-1}(A\cap B)=0
        \end{tikzcd}
    \]
    Since $H_n(B)$ also vanishes, the map $j$ is hence the map induced by the inclusion $\iota_A:A\to E_{n+1}(G)$. However, it can be seen that $\iota_A$ is homotopic to a constant map as follows: Firstly we can deform $A$ to $E_n(G)$ via a deformation retraction
    \begin{align*}
        h_1: A \bprod I_{n+3} & \la  A                  \\
        ((g,k,x),t)           & \lm (g,\min(k,n+3-t),x)
    \end{align*}
    and hence deform $\iota_A$ to the inclusion $\iota_{E_n(G)}: E_n(G) \hookrightarrow E_{n+1}(G)$. This inclusion is homotopic to a constant map via:
    \begin{align*}
        h_2:E_{n}(G) \bprod I_{n+3} & \la E_{n+1}(G) \\
        (x,t)                       & \lm (e,t,x).
    \end{align*}
    It follows that in the Mayer-Vietoris sequence the map $j:H_n(A)\oplus H_n(B)\to H_n(E_{n+1}(G))$ is the $0$-map, by the inductive assumption $H_{n-1}(A\cap B)=0$ and therefore $H_n(E_{n+1}(G))=0$.

    We have now shown that $E_{n+1}(G)$ is simply connected and the first $n$ homology groups vanish, hence by the Hurewicz theorem for graphs it follows that $E_{n+1}(G)$ is $n$-connected, completing the induction.
\end{proof}

\begin{theorem}\label{thm:BG}
    The graph $EG$ is weakly contractible and the quotient $BG:=G\backslash EG$ is a $K(G,1)$-graph.
\end{theorem}

\begin{proof}
    Applying \ref{lem:colimithomotopygroups} to the sequence $X_k\letbe E_{k+1}(G)$, whose terms satisfy $\AA_i(X_k)=0$ for all $i\leq k$ by \ref{prop:En_connected} it follows that $\AA_i(EG)=0$ for all $i\geq 1$, i.e.\ $EG$ is weakly contractible. Since the action of $G$ on $EG$ is freely sparse, the quotient map $EG\to BG$ is a covering with group $G$. By \ref{lem:FundGrpCovering}, it follows that $\AA_1(BG)\cong G$ and by \ref{prop:UditChrisCovering} it follows that the higher homotopy groups of $BG$ are isomorphic to those of $EG$ and hence vanish. Thus $BG$ is a $K(G,1)$-graph.
\end{proof}
\begin{remark}
    Though in principle, the discrete homology groups of $BG$ could be used to compute the group homology of $G$, the graph $BG$ has infinitely many vertices. Recently there have been substantial advances in the computability of discrete homology groups of finite graphs (c.f. \cite{Ender:First,Ender:Higher,Nathan:Red,Nathan:TDA}), computing discrete homology groups of graphs remains very hard, especially for infinite graphs.
\end{remark}
\subsection{Construction of $K(M,n)$}
\begin{theorem}[Blakers-Massey theorem for graphs]\label{thm:blakersmassey}
    Let $X$ be a graph and $A,B\subseteq X$ be induced subgraphs with non-empty intersection such that $X=A\cup B$. Furthermore, let $m,n\geq 1$ such that the distance between $A\setminus B$ and $B\setminus A$ is at least $m+n$. If it holds that
    \begin{align*}
        \AA_i(A,A\cap B,*)=0 & ~\for 0<  i< m, \text{ and} \\
        \AA_i(B,A\cap B,*)=0 & ~\for 0<  i< n
    \end{align*}
    for all basepoints $*\in A\cap B$, then it holds that the map induced by the inclusion $$\iota_*: \AA_k(A,A\cap B,*)\to \AA_k(X,B,*)$$ is bijective for $1\leq k< m+n-2$ and surjective for $k=m+n-2$.
\end{theorem}
\begin{proof}
    Firstly, observe that in the $(n+m-1)$-skeleton $|\sk_{n+m-1}N_1(X)|$, it holds that $$|\sk_{n+m-1}N_1(X)|=|\sk_{n+m-1}N_1(A)|\cup |\sk_{n+m-1}N_1(B)|.$$ This is because the distance between $A\setminus B$ and $B\setminus A$ is at least $m+n$, so any singular $k$-cube $I_1^k\to X$ factors through $A$ or $B$ for $k\leq m+n-1$. Since for CW-complexes, the $k$-th homotopy groups only depend on the $(k+1)$-skeleton for all $k$, we obtain a commutative diagram
    \[
        \begin{tikzcd}
            \Graph \ar[d,"\AA_k"] \ar[r,"N_1"] &\cSet \ar[r,"\sk_{n+m-1}"] &\cSet\ar[d,"|\cdot|"]\\
            \Grp & & \Top \ar[ll,"\pi_k"]
        \end{tikzcd}
    \]
    for all $k\leq m+n-2$. The claim then follows, using \ref{lem:relativemagic}, from the classical Blakers-Massey theorem (cf.\ \cite[6.4.1]{Dieck:AT}) applied to the spaces $|\sk_{n+m-1}N_1(A)|$, $|\sk_{n+m-1}N_1(B)|$ and $|\sk_{n+m-1}N_1(X)|$. Note that the assumption in \cite[6.4.1]{Dieck:AT} is that the subspaces are open; however, for a CW-pair $(Y,A)$, there is an open neighborhood $U$ of $A$ in $Y$ such that $A$ is a deformation retract of $U$ (cf.\ \cite[5.4.4]{Dieck:AT}).
\end{proof}

\begin{corollary}[Freudenthal suspension for graphs]\label{cor:Freudenthal}
    Let $X$ be an $(n-1)$-connected graph and let $l\geq 2n+2$, then the inclusion
    \begin{align*}
        \iota_1\from X & \la \Sigma_l X \\
        v              & \lm (v,1)
    \end{align*}
    gives rise to isomorphisms $\AA_i(X,*)\to \AA_{i+1}(\Sigma_l X,*)$ for all $1\leq i \leq 2n-2$ and a surjection for $i=2n-1$.
\end{corollary}
\begin{proof}
    Recall that $\Sigma_l X$ is the graph
    \[
        \Sigma_l X = (X\bprod I_l)/\sim,
    \]
    where $\sim$ is the equivalence relation generated by $(v,0)\sim (v',0)$ and $(v,l)\sim (v',l)$ for all $v,v'\in V(X)$. Let $t$ be the equivalence class of $(v,0)$ and $b$ be the equivalence class of $(v,l)$ and define $A$ to be the induced subgraph of $\Sigma_l X$ obtained by removing the vertex $t$ and $B$ to be the induced subgraph obtained by removing the vertex $b$. The distance between $A\setminus B$ and $B\setminus A$ is the distance between $t$ and $b$ which is $l$. Note that $A\cap B$ is homotopy equivalent to $X$ and hence $(n-1)$-connected. Furthermore, $A$ and $B$ are contractible, therefore, it follows from the long exact sequence of a pair that $\AA_i(A,A\cap B)=0$ and $\AA_i(B,B\cap A)=0$ for all $0< i< n+1$. We can hence apply \ref{thm:blakersmassey} and obtain that $\iota_*: \AA_k(A,A\cap B,*)\to \AA_k(\Sigma_l X,B,*)$ is an isomorphism for $1<k<2n$ and a surjection for $k=2n$. Now, it follows using the long exact sequence of pairs that
    \[
        \AA_i(X)\cong \AA_i(A\cap B) \cong \AA_{i+1}(A,A\cap B)\cong \AA_{i+1}(\Sigma_l X,B)\cong \AA_{i+1}(\Sigma_l X)~ \for ~ 1\leq i \leq 2n-2.
    \]
    In degree $2n-1$, the same composition yields a surjection, since the map $\AA_{2n}(\Sigma_l X,B,*)\to \AA_{2n}(\Sigma_l X,*)$ is a surjection.
\end{proof}
\begin{proposition}[killing homotopy groups]\label{prop:killinghomotopygroups}~
    \begin{enumerate}
        \item Let $(X,*_X)$ be a connected based graph and $\alpha\in\AA_k(X,*_X)$ be a homotopy class.
              Furthermore, let $f\from (\del I_l^{k+1},*_{\del I_l^{k+1}})\to (X,*_X)$ with $l\geq k+1$ be a representative of $\alpha$. Consider the graph
              \[
                  \hat{X}\letbe\left(X\sqcup\left(\del I_l^{k+1}\bprod I_{k+2}\right)\right)/\sim
              \]
              where $\sim$ is the equivalence relation generated by $(v,0)\sim f(v)$ for all $v\in \del I_l^{k+1}$ and $(v,k+2)\sim (w,k+2)$ for all $v\in \del I_l^{k+1}$. Then the inclusion $\iota\from X\to \hat{X}$ is a $k$-equivalence and $\iota_*$ induces an epimorphism $\faktor{\AA_k(X,*_X)}{\langle\langle \alpha\rangle\rangle}\to \AA_k(\hat{X},*_X)$ (where $\langle\langle \alpha\rangle\rangle$ denotes the normal closure for $k=1$, and $\langle \alpha\rangle$ for $k\geq 2$).
        \item Let $(X,*_X)$ be a connected based graph, then there is a connected based graph $(\hat{X},*_{\hat{X}})$ and an inclusion $\iota\from (X,*_X)\to (\hat{X},*_{\hat{X}})$ such that $\iota$ is an $n$-equivalence and $\AA_n(\hat{X},*_{\hat{X}})=0$.
    \end{enumerate}
\end{proposition}
\begin{proof}
    We prove $(1)$ first, let $f,\alpha$ and $\hat{X}$ be as above.
    Consider the induced subgraphs
    \[
        A\letbe \left(X\sqcup\left(\del I_l^{k+1}\bprod I_{k+1}\right)\right)/\sim \text{ and } B\letbe \hat{X}\setminus X,
    \]
    and let $b\in A\cap B$ be a basepoint. Note that $A$ is obtained from $\hat{X}$ by removing the vertex at the ``tip'' of the cone attached to $X$. The following map is a deformation retraction of $A$ onto $X$:
    \begin{align*}
        A\bprod I_{k+1} & \la A                                                                                   \\
        (v,t)           & \lm \begin{cases}
                                  v                       & \text{ for } v\in X                                       \\
                                  (v_1,\min(v_2,(k+1-t))) & \text{ for } v=(v_1,v_2)\in \del I_l^{k+1}\bprod I_{k+1}.
                              \end{cases} \\
    \end{align*}
    $B$ on the other hand is a cone over $\del I_l^{k+1}$ and hence contractible and the intersection $A\cap B$ is homotopy equivalent to $\del I_l^{k+1}$ and hence $(k-1)$-connected. It follows from the long exact sequence of the pair $(B,A\cap B)$ that
    \[
        \AA_i(B,A\cap B,*)=0~\for 0<  i< k+1.
    \]
    Also, note that the distance from $A\setminus B$ to $B\setminus A$ is $k+2$. We can hence apply \ref{thm:blakersmassey} with $m=1$ and $n=k+1$ to obtain that the inclusion $\iota\from (B,A\cap B) \to (\hat{X},A)$ induces an isomorphism $\AA_i(B,A\cap B)\to \AA_i(\hat{X},A)$ for $1<i<k$ and a surjection for $i=k$.
    Now consider the map $\iota$ between the long exact sequences of the pairs $(B,A\cap B)$ and $(\hat{X},A)$ using that $B$ is contractible:
    \[
        \begin{tikzcd}[row sep=large]
            \cdots \arrow[r]
            & 0 \arrow[r] \arrow[d, "\iota_*"]
            & \AA_{k+1}(B, A \cap B) \arrow[r, "\delta"] \arrow[d, "\iota_*"]
            & \AA_k(A \cap B) \arrow[r] \arrow[d, "\iota_*"]
            & 0 \arrow[r] \arrow[d, "\iota_*"]
            & 0 \arrow[r] \arrow[d, "\iota_*"]
            & \cdots \\
            \cdots \arrow[r]
            & \AA_{k+1}(\hat{X}) \arrow[r]
            & \AA_{k+1}(\hat{X}, A) \arrow[r, "\delta"]
            & \AA_k(A) \arrow[r]
            & \AA_k(\hat{X}) \arrow[r]
            & 0 \arrow[r]
            & \cdots
        \end{tikzcd}
    \]
    The image of the homotopy class of the identity map $[\id{\del I_l^{k+1}}]\in \AA_k(\del I_l^{k+1})\cong \AA_k(A\cap B)$ under the map $\iota_*\from \AA_k(A\cap B)\to \AA_k(A)$ is $\alpha$; therefore, since the upper boundary map is an isomorphism, the image of $\alpha$ under the inclusion $\AA_k(A)\to \AA_k(\hat{X})$ has to be trivial.
    This completes the proof of $(1)$.

    Consider a generating set $I$ of $\AA_n(X,*_X)$. Using the axiom of choice, we choose a representative for every $\alpha\in I$ and denote by $X_\alpha$ the graph obtained by attaching a cone to $X$ along this representative of $\alpha$ as in $(1)$. Let
    \[\mathcal{Y}\letbe \left\{X_J=X\cup \bigcup_{\alpha\in J} X_\alpha \,\middle| \, J\subseteq I \text{ with }\#J<\infty \right\}~\text{ and }~ \hat{X}\letbe\bigcup_{\alpha\in I} X_\alpha\]

    For all $Y\in \mathcal{Y}$, since $Y$ is obtained from $X$ by attaching a finite number of cones to $X$, the inclusion $X\hookrightarrow Y$ is an $n$-equivalence. We will show that also the inclusion $\iota\from X\to \hat{X}$ is an $n$-equivalence.

    Consider a representative $f$ of a homotopyclass $\beta\in \AA_k(\hat{X},*_X)$ with $k\leq n$. Since $f$ has finite image, $f$ must be contained in some $Y\in \mathcal{Y}$. The inclusion $\iota\from X\to Y$ is an $n$-equivalence for all $Y\in \mathcal{Y}$, therefore $\beta$ lies in the image of $\iota_k:\AA_k(X,*)\to \AA_k(\hat{X},*)$, thus $\iota_k$ is surjective for all $1\leq k\leq n$. Now let $k<n$. If $\iota(\beta)$ gets killed in $\hat{X}$, then it has to be killed in a $Y\in \mathcal{Y}$, since every a.e. constant homotopy has finite image. Again, since the inclusion $X\to Y$ is an $n$-equivalence, $\beta$ has to be trivial to begin with. We have hence shown that $\iota$ is an $n$-equivalence.

    Lastly, the inclusion $\iota_n$ factors, for every generator $\alpha\in I$ as
    \[
        X\hookrightarrow X_\alpha \hookrightarrow \hat{X}.
    \]
    Since $\alpha$ gets killed in $X_\alpha$, it follows that $\iota_n(\alpha)$ is trivial and therefore, since $\iota_n$ is surjective, $\AA_n(\hat{X},*)=0$.

\end{proof}
\begin{theorem}[partial Postnikov systems]\label{thm:postnikov}
    Let $(X,*_X)$ be a connected based graph and $n\geq 1$. Then there is a sequence of based graphs $(X_j,*_{X_j})$, $0\leq j\leq n$ and based graph maps $p_j:(X_j,*_{X_j})\to (X_{j-1},*_{X_{j-1}})$ and $\vf_j:X\to X_j$ fitting in a commutative diagram
    \[
        \begin{tikzcd}[row sep=large, column sep=large]
            X \arrow[d,swap, "\vf_n"']
            \arrow[dr, "\vf_{n-1}"]
            \arrow[drrr, bend left=9, "\vf_1"]
            \arrow[drrrr, bend left=12, "\vf_0"] \\
            X_n \arrow[r, "p_n"]
            & X_{n-1} \arrow[r, "p_{n-1}"]
            & \cdots \arrow[r, "p_2"]
            & X_1 \arrow[r, "p_1"]
            & X_0 = \{*\}
        \end{tikzcd}
    \]
    such that
    \begin{enumerate}
        \item $\vf_j$ induces an isomorphism $\AA_i(X)\to \AA_i(X_j)$ for all $0 \leq i\leq j\leq n$ and
        \item $\AA_i(X_j,*_{X_j})=0$ for all $i>j$.
        \item The maps $p_j: (X_j,*_{X_j})\to (X_{j-1},*_{X_{j-1}})$ are fibrations.
    \end{enumerate}
\end{theorem}
\begin{proof}
    Given $n$ and $X$ as above, we firstly construct an inclusion $\iota_{n+1}:X\to H^n_{n+1}$ using \ref{prop:killinghomotopygroups} with $\AA_{n+1}(H^n_{n+1})=0$ such that $\iota_*$ induces isomorphisms $\AA_i(H^n_{n+1})\cong \AA_i(X)$ for all $i\leq n$. From here, using \ref{prop:killinghomotopygroups} again, we inductively construct $k$-equivalences $\iota_k: H^n_{k-1}\to H^n_{k}$ for $k>n+1$ with $\AA_{k}(H^n_{k})=0$. Define
    \[
        K_n\letbe \colim{}\left(X\hookrightarrow H^n_{n+1}\hookrightarrow H^n_{n+2}\hookrightarrow \cdots \right).
    \]
    By \ref{lem:colimithomotopygroups}, the universal cocone map $\hat{\vf}_n:X\to K_n$ induces isomorphisms $\AA_i(X)\cong \AA_i(K_n)$ for all $i\leq n$, and, also by \ref{lem:colimithomotopygroups}, $\AA_i(K_n)=0$ for all $i>n$. From here, we can inductively over $j$ construct $\hat{p}_j:K_j \to K_{j-1}$ by the same procedure: Define
    \[
        K_{j-1}\letbe \colim{}\left(K_j\hookrightarrow H^{j-1}_{j}\hookrightarrow H^{j-1}_{j+1}\hookrightarrow \cdots \right),
    \]
    and $\hat{p}_j$ as the universal cocone map. However, the maps $\hat{p}_k$ are not fibrations yet. Let $X_0=\{*\}$ and let $s_0:K_0\to X_0$ be the universal map. We inductively construct $X_{j+1}$ as the mapping path graph of the composition $s_j\circ \hat{p}_{j+1}$, and $s_{j+1}:K_{j+1}\xrightarrow{\sim} E_{s_j\circ \hat{p}_{j+1}},v\mapsto(c_{s_j\circ \hat{p}_{j+1}(v)},v)$ as the weak equivalence, and $p_{j+1}$ the fibration that comes with the mapping path graph construction.
    \[
        \begin{tikzcd}
            K_n\arrow[r,"\hat{p}_n"]&K_{n-1}\arrow[r,"\hat{p}_{n-1}"]&\cdots\arrow[r] &K_j\arrow[r,"\hat{p}_j"]\arrow[d,"s_j","\sim"']&K_{j-1}\arrow[d,"s_{j-1}","\sim"']\arrow[r,"\hat{p}_{j-1}"]&\cdots\arrow[r,"\hat{p}_1"]&K_0\ar[d,"s_0","\sim"']\\
            &&& X_j\letbe E_{s_{j-1}\circ \hat{p}_{j}}\arrow[r,"p_j"]&X_{j-1}\arrow[r,"p_{j-1}"]&\cdots\arrow[r,"p_1"]&X_0
        \end{tikzcd}
    \]
    Finally, we set $\vf_n\letbe s_n\circ \hat{\vf}_n\from X\to X_n$, and for $j<n$ the maps $\vf_j:X\to X_j$ are defined as the composition of $\vf_n$ with the maps $p_{j+1}\circ \cdots \circ p_{n-1}\circ p_{n}$, finishing the conditions of the theorem.
\end{proof}
\begin{construction}[Eilenberg-MacLane graphs]\label{const:KPIN}
    Let $M$ be an Abelian group and $n\geq 2$. We construct a connected graph $K(M,n)$ such that $\AA_n(K(M,n))\cong M$ and $\AA_i(K(M,n))=0$ for all $i\neq n$.
    \begin{enumerate}
        \item Let $X = BM$ be a connected graph such that $\AA_1(X,*)\cong M$; this exists by \ref{thm:BG}.
        \item Then, consider the $(n-1)$-fold suspension $\Sigma_{n+2}^{n-1} X$,
        \item and build the partial Postnikov system of length $n$.
    \end{enumerate}
    The $n$-th piece of the tower $(\Sigma_{n+2}^{n-1} X)_n$ is a $K(M,n)$-graph.
\end{construction}
\begin{proof}
    Given $X$ as in the construction, firstly observe that, since $M$ is Abelian, $H_1(X)=M$, by the Hurewicz theorem for graphs. Iteratively applying the Mayer-Vietoris sequence we obtain that
    \[
        H_k(\Sigma_{n+2}^{n-1} X) = 0 ~\for 0< k < n \text{ and } H_n(\Sigma_{n+2}^{n-1} X)=M.
    \]
    It now follows from the Hurewicz theorem for graphs that $\AA_k(\Sigma_{n+2}^{n-1} X)=0$ for $0< k < n$ and $\AA_n(\Sigma_{n+2}^{n-1} X)\cong M$. Since the Postnikov tower kills the higher homotopy groups, the $n$-th term $(\Sigma_{n+2}^{n-1} X)_n$ is a $K(M,n)$.
\end{proof}
\begin{corollary}\label{cor:sequence_homotopy_groups}
    For every sequence of groups $M_1,M_2,\ldots$ such that $M_n$ is Abelian for all $n\geq 2$, there is a connected graph $X$ such that $\AA_n(X)\cong M_n$ for all $n\geq 1$.
\end{corollary}
\begin{proof}
    Consider the sequence of based graphs $(X_1,*_1),(X_2,*_2),\ldots$ where $X_1 = BM_1$ from \ref{thm:BG} and $X_n$ is a $K(M_n,n)$ from \ref{const:KPIN} for $n\geq 2$, and consider the colimit
    \[
        X=\colim{n}\left(X_1\xrightarrow{\id{X_1}\times c_{*_2}} X_1\times X_2\xrightarrow{\id{X_1}\times\id{X_2} \times c_{*_3}} \cdots \right).
    \]
    For every pair of based graphs $(X,*_X)$ and $(Y,*_Y)$ it holds that $\AA_n(X \times Y, (*_X, *_Y)) \cong \AA_n(X, *_X) \times \AA_n(Y, *_Y)$(cf. \cite{Chris:Revisited}). Since every $X_n$ is $(n-1)$-connected, the maps $\id{X_1}\times \cdots \times \id{X_{n-1}}\times c_{*_n}$ are $(n-1)$-equivalences. By \ref{lem:colimithomotopygroups}, we can then conclude that
    \[
        \AA_n(X)=\AA_n(X_1\times \cdots \times X_n)\cong\AA_n(X_1)\times \cdots \times \AA_n(X_n)\cong M_n,
    \]
    as desired.
\end{proof}
\begin{theorem}[Whitehead Towers]\label{thm:whitehead}
    Let $(X,*_X)$ be a connected based graph. Then there exists a sequence of based graphs $(X^{(i)},*_i)$, $i\geq 0$, and graph maps $\psi_i$
    \[
        \cdots \xrightarrow{\psi_3} X^{(2)} \xrightarrow{\psi_2} X^{(1)} \xrightarrow{\psi_1} X^{(0)}=X
    \]
    such that
    \begin{enumerate}
        \item $\AA_i(X^{(j)})=0$ for $0<i\leq j$,
        \item $\psi_j$ induces isomorphisms $\AA_i(X^{(j)})\to \AA_i(X^{(j-1)})$ for all $i>j$,
        \item $\psi_j$ is a fibration with fiber $K(\AA_j(X),j-1)$.
    \end{enumerate}
\end{theorem}
\begin{proof}
    We construct the sequence inductively. Let $X^{(0)}=X$ and let $X^{(1)}$ be the universal cover with covering map $\psi_1: X^{(1)} \to X$. Then $\psi_1$ induces an isomorphism on higher homotopy groups by \ref{prop:UditChrisCovering} and the fiber is discrete and can be endowed with a group structure isomorphic to $\AA_1(X)$ by \ref{lem:FundGrpCovering}.
    Now, assume we have constructed $(X^{(j)},\psi_j)$ as desired for all $j\leq n-1$ and consider the partial Postnikov tower for $X^{(n-1)}$:
    \[
        \begin{tikzcd}[row sep=large, column sep=large]
            X^{(n-1)} \arrow[d,swap, "\vf_n"']
            \arrow[dr, "\vf_{n-1}"]
            \arrow[drrr, bend left=9, "\vf_1"]
            \arrow[drrrr, bend left=12, "\vf_0"] \\
            \left(X^{(n-1)}\right)_{n} \arrow[r, "p_n"]
            & \left(X^{(n-1)}\right)_{n-1} \arrow[r, "p_{n-1}"]
            & \cdots \arrow[r, "p_2"]
            & \left(X^{(n-1)}\right)_1 \arrow[r, "p_1"]
            & \left(X^{(n-1)}\right)_0=\{*\}
        \end{tikzcd}
    \]
    Note that for $i<n$, $\left(X^{(n-1)}\right)_i$ is weakly contractible and $\left(X^{(n-1)}\right)_n$ is a $K(\AA_n(X),n)$.
    Let
    \begin{align*}
        X^{(n)} & \letbe\mathrm{hofib}(\vf_n)                                                                                                                                                                           \\
                & =\left\{\,(\omega,x)\in P\left(\left(X^{(n-1)}\right)_{n}\right)\times \left(X^{(n-1)}\right) \,\middle|\,\lim{m\to \infty}\omega(m) = \vf_n(x),\, \omega(0)=*_{\left(X^{(n-1)}\right)_n} \,\right\},
    \end{align*}
    the homotopy fiber of $\vf_n$. By \ref{thm:postnikov}, $\vf_n$ induces an isomorphism on homotopy groups $\AA_i(X^{(n-1)},*)\cong \AA_i(\left(X^{(n-1)}\right)_n,*)$ for $i\leq n$, it thus follows from the long exact sequence of a fibration, that
    \[
        \AA_i(X^{(n)})=\begin{cases} 0 & i\leq n \\ \AA_i\left(X\right) & i>n . \end{cases}
    \]
    Thus by setting
    \begin{align*}
        \psi_n \from X^{(n)} & \la X^{(n-1)} \\
        (\omega,x)           & \lm x
    \end{align*}
    we have constructed the $n$-th stage of the tower as desired.
\end{proof}

\bibliographystyle{alpha}
\bibliography{ref}
\end{document}

%% file: Makros.tex
\def\letbe{\coloneqq}

\def\from{\colon}

\newcommand{\colim}[1]{\underset{#1}{\operatorname{colim}~}}

\renewcommand{\lim}[1]{\underset{#1}{\operatorname{lim}}~}

\newcommand{\id}[1]{\operatorname{id}_{#1}}

\newcommand{\pr}[1]{\operatorname{pr}_{#1}}

\newcommand{\Umg}[2]{\mathcal{U}_#1(#2)}

\def \bprod{\mathrel{\square}}

\newcommand{\join}[1]{\mathrel{\star}_{#1}}

\def\im{\mathrm{im}}

\def\sk{\mathbf{sk}}

\tikzcdset{
    ar/.code={\arrow[#1]}
}

\def \la{\longrightarrow}

\def \lm{\longmapsto}

\newenvironment{tikzeq*}
{
    \begingroup
    \begin{equation*}
        \begin{tikzpicture}[baseline=(current bounding box.center)]
            }
            {
        \end{tikzpicture}
    \end{equation*}
    \endgroup
    \ignorespacesafterend
}

\tikzset
{
diagram/.style=
    {
        matrix of math nodes,
        column sep={4.3em,between origins},
        row sep={4em,between origins},
        text height=1.5ex,
        text depth=.25ex
    },
over/.style={preaction={draw=white,-,line width=6pt}},
every to/.style={font=\footnotesize},
mono/.style={right hook->},
epi/.style={-{Latex[open]}},
cof/.style={>->},
fib/.style={->>},
}

\newcommand{\fa}[1]{{\texttt{#1}}}

\def\for{\textnormal{for }}

\def\ve{\varepsilon}
\def\vf{\varphi}
\def\del{\partial}
\def\NN{\mathbb{N}}
\def\ZZ{\mathbb{Z}}

\def\one{[\mathbf{1}]}

\def\horn{\scalebox{1.15}{$\sqcap$}}

\def\cd{d}
\def\cs{s}
\def\cg{c}

\def\csd{\del}
\def\css{\sigma}
\def\csg{\Gamma}

\def\gd{\mathfrak{d}}
\def\gs{\mathfrak{s}}
\def\gg{\mathfrak{c}}

\newcommand{\cat}[1]{{\boldsymbol{\mathsf{#1}}}}

\def\Graph{\cat{Graph}}

\def\cSet{\cat{cSet}}

\def\Set{\cat{Set}}

\def\Ab{\cat{Ab}}

\def\Grp{\cat{Grp}}

\def\Top{\cat{Top}}

\DeclareFontFamily{U}{dmjhira}{}
\DeclareFontShape{U}{dmjhira}{m}{n}{ <-> dmjhira }{}

\newcommand{\E}[1]{E_{#1}}

\newcommand{\hg}[2]{\vert\horn^{#2}\vert_{#1}}

\newcommand{\hgpar}[3]{TBD}

\def\inc{\eta}

%% file: MakroThmCounter.tex
\makeatletter
\newcommand{\numberlike}[2]{%
   \expandafter\def\csname c@#1\endcsname{%
      \expandafter\csname c@#2\endcsname}%
}
\makeatother

\def\DefaultNumberTheoremWithin{section}

\theoremstyle{plain}
\newtheorem{lemma}{Lemma}
\numberwithin{lemma}{\DefaultNumberTheoremWithin}
\labelformat{lemma}{Lemma~#1}

\numberwithin{claim}{\DefaultNumberTheoremWithin}
\numberlike{claim}{lemma}
\labelformat{claim}{Claim~#1}
\newtheorem{theorem}{Theorem}
\numberwithin{theorem}{\DefaultNumberTheoremWithin}
\numberlike{theorem}{lemma}
\labelformat{theorem}{Theorem~#1}
\newtheorem{corollary}{Corollary}
\numberwithin{corollary}{\DefaultNumberTheoremWithin}
\numberlike{corollary}{lemma}
\labelformat{corollary}{Corollary~#1}
\newtheorem{construction}{Construction}
\numberwithin{construction}{\DefaultNumberTheoremWithin}
\numberlike{construction}{lemma}
\labelformat{construction}{Construction~#1}
\newtheorem{proposition}{Proposition}
\numberwithin{proposition}{\DefaultNumberTheoremWithin}
\numberlike{proposition}{lemma}
\labelformat{proposition}{Proposition~#1}

\numberwithin{conjecture}{\DefaultNumberTheoremWithin}
\numberlike{conjecture}{lemma}
\labelformat{conjecture}{Conjecture~#1}

\newtheorem*{warning*}{Warning}

\theoremstyle{definition}
\newtheorem{definition}{Definition}
\numberwithin{definition}{\DefaultNumberTheoremWithin}
\numberlike{definition}{lemma}
\labelformat{definition}{Definition~#1}

\theoremstyle{definition}

\numberwithin{question}{\DefaultNumberTheoremWithin}
\numberlike{question}{lemma}
\labelformat{question}{Question~#1}

\theoremstyle{definition}

\numberwithin{problem}{\DefaultNumberTheoremWithin}
\numberlike{problem}{lemma}
\labelformat{problem}{Problem~#1}

\theoremstyle{definition}
\newtheorem{remark}{Remark}
\numberwithin{remark}{\DefaultNumberTheoremWithin}
\numberlike{remark}{lemma}
\labelformat{remark}{Remark~#1}
\theoremstyle{remark}

\newtheorem{example}{Example}
\numberwithin{example}{\DefaultNumberTheoremWithin}
\numberlike{example}{lemma}
\labelformat{example}{Example~#1}

\labelformat{equation}{#1}
\labelformat{figure}{Figure~#1}
\labelformat{table}{Table~#1}
\labelformat{chapter}{Chapter~#1}
\labelformat{appendix}{Appendix~#1}
\labelformat{section}{Section~#1}
\labelformat{subsection}{Subsection~#1}

%% file: ref.bib
@article{Chris:Revisited,
  title   = {Cubical setting for discrete homotopy theory, revisited},
  volume  = {160},
  doi     = {10.1112/S0010437X24007486},
  number  = {12},
  journal = {Compositio Mathematica},
  author  = {Daniel Carranza and Chris Kapulkin},
  year    = {2024},
  pages   = {2856–2903}
}

@article{Udit:Fundamentalgroupoid,
  author   = {Kapulkin, Chris and Mavinkurve, Udit},
  title    = {The fundamental group in discrete homotopy theory},
  journal  = {Adv. in Appl. Math.},
  fjournal = {Advances in Applied Mathematics},
  volume   = {164},
  year     = {2025},
  pages    = {Paper No. 102838, 56},
  issn     = {0196-8858,1090-2074},
  mrclass  = {55Q99 (05C38 05C63)},
  mrnumber = {4847138},
  doi      = {10.1016/j.aam.2024.102838 },
  url      = {https://doi.org/10.1016/j.aam.2024.102838}
}

@article{Monoidalproducts,
  author     = {Kapulkin, Chris and Kershaw, Nathan},
  title      = {Closed symmetric monoidal structures on the category of
                graphs},
  journal    = {Theory Appl. Categ.},
  fjournal   = {Theory and Applications of Categories},
  volume     = {41},
  year       = {2024},
  pages      = {Paper No. 23, 760--784},
  issn       = {1201-561X},
  mrclass    = {18M05 (05C76)},
  mrnumber   = {4774716},
  mrreviewer = {Marko\ \v Zivkovi\'c}
}

@article{Helene:Homology,
  author     = {Barcelo, H\'el\`ene and Capraro, Valerio and White, Jacob A.},
  title      = {Discrete homology theory for metric spaces},
  journal    = {Bull. Lond. Math. Soc.},
  fjournal   = {Bulletin of the London Mathematical Society},
  volume     = {46},
  year       = {2014},
  number     = {5},
  pages      = {889--905},
  issn       = {0024-6093,1469-2120},
  mrclass    = {55N35 (54E35)},
  mrnumber   = {3262192},
  mrreviewer = {Barbara\ Di Fabio},
  doi        = {10.1112/blms/bdu043},
  url        = {https://doi.org/10.1112/blms/bdu043}
}

@book{Dieck:AT,
  author    = {tom Dieck, Tammo},
  title     = {Algebraic topology},
  fseries   = {EMS Textbooks in Mathematics},
  series    = {EMS Textb. Math.},
  isbn      = {978-3-03719-048-7},
  year      = {2008},
  publisher = {Z{\"u}rich: European Mathematical Society (EMS)},
  language  = {English},
  doi       = {10.4171/048},
  zbmath    = {5318519},
  zbl       = {1156.55001}
}

@article{Helene:Vanishing,
  author   = {Barcelo, H{\'e}l{\`e}ne and Greene, Curtis and Jarrah, Abdul Salam and Welker, Volkmar},
  title    = {On the vanishing of discrete singular cubical homology for graphs},
  fjournal = {SIAM Journal on Discrete Mathematics},
  journal  = {SIAM J. Discrete Math.},
  issn     = {0895-4801},
  volume   = {35},
  number   = {1},
  pages    = {35--54},
  year     = {2021},
  language = {English},
  doi      = {10.1137/20M1338484},
  zbmath   = {7325839},
  zbl      = {1459.05054}
}

@article{Helene:Foundations,
  author   = {Barcelo, H{\'e}l{\`e}ne and Kramer, Xenia and Laubenbacher, Reinhard and Weaver, Christopher},
  title    = {Foundations of a connectivity theory for simplicial complexes},
  fjournal = {Advances in Applied Mathematics},
  journal  = {Adv. Appl. Math.},
  issn     = {0196-8858},
  volume   = {26},
  number   = {2},
  pages    = {97--128},
  year     = {2001},
  language = {English},
  doi      = {10.1006/aama.2000.0710},
  zbmath   = {1598991},
  zbl      = {0984.57014}
}

@article{Lutz:Higher,
  author   = {Lutz, Bob},
  title    = {Higher discrete homotopy groups of graphs},
  fjournal = {Algebraic Combinatorics},
  journal  = {Algebr. Comb.},
  issn     = {2589-5486},
  volume   = {4},
  number   = {1},
  pages    = {69--88},
  year     = {2021},
  language = {English},
  doi      = {10.5802/alco.151},
  zbmath   = {7329421},
  zbl      = {1460.05207}
}

@article{Welker:Connections,
  author   = {Barcelo, H{\'e}l{\`e}ne and Greene, Curtis and Jarrah, Abdul Salam and Welker, Volkmar},
  title    = {Homology groups of cubical sets with connections},
  fjournal = {Applied Categorical Structures},
  journal  = {Appl. Categ. Struct.},
  issn     = {0927-2852},
  volume   = {29},
  number   = {3},
  pages    = {415--429},
  year     = {2021},
  language = {English},
  doi      = {10.1007/s10485-020-09621-x},
  zbmath   = {7379925},
  zbl      = {1476.55019}
}

@article{Kapulkin:Homotopygroups,
  author   = {Carranza, Daniel and Kapulkin, Chris},
  title    = {Homotopy groups of cubical sets},
  fjournal = {Expositiones Mathematicae},
  journal  = {Expo. Math.},
  issn     = {0723-0869},
  volume   = {41},
  number   = {4},
  note     = {Id/No 125518},
  year     = {2023},
  language = {English},
  doi      = {10.1016/j.exmath.2023.125518},
  zbmath   = {7799885},
  zbl      = {1537.55016}
}

@article{Babson:CubicalApprox,
  author   = {Babson, Eric and Barcelo, H{\'e}l{\`e}ne and de Longueville, Mark and Laubenbacher, Reinhard},
  title    = {Homotopy theory of graphs},
  fjournal = {Journal of Algebraic Combinatorics},
  journal  = {J. Algebr. Comb.},
  issn     = {0925-9899},
  volume   = {24},
  number   = {1},
  pages    = {31--44},
  year     = {2006},
  language = {English},
  doi      = {10.1007/s10801-006-9100-0},
  zbmath   = {5046867},
  zbl      = {1108.05030}
}

@misc{Chris:ntypes,
  author       = {Daniel Carranza and Chris Kapulkin},
  title        = {Discrete homotopy hypothesis for n-types},
  year         = {2026},
  howpublished = {Preprint, {arXiv}:2602.19293 [math.{AT}] (2026)},
  url          = {https://arxiv.org/abs/2602.19293},
  arxiv        = {arXiv:2602.19293}
}

@article{Hardeman:Graphcoverings,
  author   = {Hardeman Morrill, Rachel},
  title    = {The lifting properties for {A}-homotopy theory},
  fjournal = {Contributions to Discrete Mathematics},
  journal  = {Contrib. Discrete Math.},
  issn     = {1715-0868},
  volume   = {19},
  number   = {3},
  pages    = {36--67},
  year     = {2024},
  language = {English},
  doi      = {10.55016/ojs/cdm.v19i3.69171},
  zbmath   = {7947067},
  zbl      = {1551.05390}
}

@article{Milnor:BundlesI,
  author   = {Milnor, John W.},
  title    = {Construction of universal bundles. {I}},
  fjournal = {Annals of Mathematics. Second Series},
  journal  = {Ann. Math. (2)},
  issn     = {0003-486X},
  volume   = {63},
  pages    = {272--284},
  year     = {1956},
  language = {English},
  doi      = {10.2307/1969609},
  zbmath   = {3118948},
  zbl      = {0071.17302}
}

@article{Milnor:BundlesII,
  author   = {Milnor, John W.},
  title    = {Construction of universal bundles. {II}},
  fjournal = {Annals of Mathematics. Second Series},
  journal  = {Ann. Math. (2)},
  issn     = {0003-486X},
  volume   = {63},
  pages    = {430--436},
  year     = {1956},
  language = {English},
  doi      = {10.2307/1970012},
  zbmath   = {3118949},
  zbl      = {0071.17401}
}

@misc{Nathan:Red,
  title         = {Discrete homology computations by reduction to zero differentials},
  author        = {Sterling Ebel and Chris Kapulkin and Nathan Kershaw},
  year          = {2026},
  eprint        = {2608.04262},
  archiveprefix = {arXiv},
  primaryclass  = {cs.CG},
  url           = {https://arxiv.org/abs/2608.04262}
}

@misc{Ender:Higher,
  title         = {Towards fast computation of higher discrete homology},
  author        = {Jacob Ender and Chris Kapulkin},
  year          = {2026},
  eprint        = {2606.00245},
  archiveprefix = {arXiv},
  primaryclass  = {cs.CG},
  url           = {https://arxiv.org/abs/2606.00245}
}

@misc{Ender:First,
  title         = {Fast computation of the first discrete homology group},
  author        = {Jacob Ender and Chris Kapulkin},
  year          = {2025},
  eprint        = {2512.13929},
  archiveprefix = {arXiv},
  primaryclass  = {cs.CG},
  url           = {https://arxiv.org/abs/2512.13929}
}

@misc{Nathan:TDA,
  title         = {Topological data analysis using persistent discrete homology},
  author        = {Chris Kapulkin and Nathan Kershaw},
  year          = {2026},
  eprint        = {2506.15020},
  archiveprefix = {arXiv},
  primaryclass  = {math.AT},
  url           = {https://arxiv.org/abs/2506.15020}
}

@article{Atkin:II,
  title    = {An algebra for patterns on a complex, II},
  journal  = {International Journal of Man-Machine Studies},
  volume   = {8},
  number   = {5},
  pages    = {483-498},
  year     = {1976},
  issn     = {0020-7373},
  doi      = {https://doi.org/10.1016/S0020-7373(76)80015-6},
  url      = {https://www.sciencedirect.com/science/article/pii/S0020737376800156},
  author   = {R.H. Atkin}
}

@article{Atkin:I,
  title    = {An Algebra for Patterns on a Complex, I},
  journal  = {International Journal of Man-Machine Studies},
  volume   = {6},
  number   = {3},
  pages    = {285-307},
  year     = {1974},
  issn     = {0020-7373},
  doi      = {https://doi.org/10.1016/S0020-7373(74)80024-6},
  url      = {https://www.sciencedirect.com/science/article/pii/S0020737374800246},
  author   = {R.H. Atkin}
}

@article{Hepworth:Magnitude,
  author   = {Hepworth, Richard and Willerton, Simon},
  title    = {Categorifying the magnitude of a graph},
  fjournal = {Homology, Homotopy and Applications},
  journal  = {Homology Homotopy Appl.},
  issn     = {1532-0073},
  volume   = {19},
  number   = {2},
  pages    = {31--60},
  year     = {2017},
  language = {English},
  doi      = {10.4310/HHA.2017.v19.n2.a3},
  zbmath   = {6826922},
  zbl      = {1377.05088}
}

@article{Helene:Perspectives,
  author   = {Barcelo, H{\'e}l{\`e}ne and Laubenbacher, Reinhard},
  title    = {Perspectives on {{\(A\)}}-homotopy theory and its applications},
  fjournal = {Discrete Mathematics},
  journal  = {Discrete Math.},
  issn     = {0012-365X},
  volume   = {298},
  number   = {1-3},
  pages    = {39--61},
  year     = {2005},
  language = {English},
  doi      = {10.1016/j.disc.2004.03.016},
  zbmath   = {2207449},
  zbl      = {1082.37050}
}

@article{Yau:Path,
  author   = {Grigor'yan, Alexander and Lin, Yong and Muranov, Yuri and Yau, Shing-Tung},
  title    = {Homotopy theory for digraphs},
  fjournal = {Pure and Applied Mathematics Quarterly},
  journal  = {Pure Appl. Math. Q.},
  issn     = {1558-8599},
  volume   = {10},
  number   = {4},
  pages    = {619--674},
  year     = {2014},
  language = {English},
  doi      = {10.4310/PAMQ.2014.v10.n4.a2},
  zbmath   = {6435208},
  zbl      = {1312.05063}
}

@article{Dochtermann:Cross,
  author   = {Dochtermann, Anton},
  title    = {Hom complexes and homotopy theory in the category of graphs},
  fjournal = {European Journal of Combinatorics},
  journal  = {Eur. J. Comb.},
  issn     = {0195-6698},
  volume   = {30},
  number   = {2},
  pages    = {490--509},
  year     = {2009},
  language = {English},
  doi      = {10.1016/j.ejc.2008.04.009},
  zbmath   = {5523101},
  zbl      = {1167.05017}
}
